\documentclass[11pt]{amsart}
\usepackage{amsfonts, bm}
\usepackage{amssymb}
\usepackage{mathrsfs}
\usepackage{hyperref}
\usepackage{amsmath}

\newtheorem{theorem}{Theorem}[section]
\newtheorem{lemma}[theorem]{Lemma}
\newtheorem{proposition}[theorem]{Proposition}
\newtheorem{corollary}[theorem]{Corollary}

\usepackage{hyperref} 
\hypersetup{
	colorlinks=true,
	linkcolor=red,
	citecolor=blue,
	filecolor=magenta,
	urlcolor=cyan          }
\allowdisplaybreaks[4]

\begin{document}
\title[Principal eigenvalue of asymmetric  nonlocal diffusion operators]{Asymptotic limit of the principal eigenvalue of asymmetric  nonlocal diffusion operators and  propagation dynamics}

\author[Y. Du, X. Fang and W. Ni]{Yihong Du$^{\dag}$,\ Xiangdong Fang$^{\ddag}$,\  Wenjie Ni$^{\dag}$}
\thanks{$^{\dag}$School of Science and Technology, University of New England, Armidale, NSW 2351, Australia}
\thanks{$^{\ddag}$School of Mathematical Sciences, Dalian University of Technology,  Dalian 116024, China}
\thanks{ \small {\bf e-mails:}  ydu@une.edu.au (Du), fangxd0401@dlut.edu.cn (Fang), wni2@une.edu.au (Ni)}

\numberwithin{equation}{section}
\date{\today}

\maketitle
\begin{abstract}
For fixed $c\in\mathbb R$, $l>0$ and a general non-symmetric kernel function $J(x)$ satisfying a standard assumption, we consider the nonlocal diffusion operator
\begin{align*}
\mathcal{L}^{J, c}_{(-l,l)}[\phi](x):=\int_{-l}^lJ(x-y)\phi(y)\,dy+c\phi'(x),
\end{align*}
and prove that its principal eigenvalue $\lambda_p(\mathcal{L}^{J, c}_{(-l,l)})$
has the following asymptotic limit:
\begin{equation*}\label{l-to-infty-c}
\lim\limits_{l\to \infty}\lambda_p(\mathcal{L}^{J, c}_{(-l,l)})=\inf\limits_{\nu\in\mathbb{R}}\big[\int_{\mathbb{R}}J(x)e^{-\nu x}\,dx+c\nu\big].
\end{equation*}
We then demonstrate how this result can be applied to determine the propagation dynamics of the associated Cauchy problem 
\begin{equation*}
	\label{cau}
	\left\{
	\begin{array}{ll}
		\displaystyle u_t = d \big[\int_{\mathbb{R}} J(x-y) u(t,y) \, dy -  u(t,x)\big] + f(u), & t > 0, \; x \in \mathbb{R}, \\
		u(0, x) = u_0(x), & x \in \mathbb{R},
	\end{array}
	\right.
\end{equation*}
with a KPP nonlinear term $f(u)$.
This provides a new approach to understand the propagation dynamics of KPP type models, very different from those based on traveling wave solutions or on the dynamical systems method of Weinberger \cite{wein1982}.

 \vskip0.23in


\vskip0.1in
\end{abstract}



\section{Introduction}

 Starting from the pioneering works of Fisher \cite{fisher} and  Kolmogorov-Petrovski\u{\i}-Piskunov (KPP) \cite{kpp}, the
 Cauchy problem
\begin{equation}
	\label{cauchy}
	\left\{
	\begin{array}{ll}
		U_t - dU_{xx} = f(U), & t > 0, \; x \in \mathbb{R}, \\
		U(0, x) = U_0(x), & x \in \mathbb{R},
	\end{array}
	\right.
\end{equation}
has been widely used to model  the propagation of species, where $U(t,x)$ stands for the population density at time $t$ and spatial location $x$ of a new or invading species, whose initial population $U_0(x)$ is  a nonnegative continuous function with nonempty compact support, to represent the assumption that initially the population exists only locally in space. Fisher \cite{fisher} assumed $f(U)=U(1-U)$ and KPP \cite{kpp} allowed more general functions $f$ but with similar behaviour, which is now known as the KPP condition characterised by
\[ {\bf (f_{KPP})}: \ \ \ \ \begin{cases} f\in C^1([0,\infty)), \mbox{ $f(0)=f(1)=0$, $f'(0)>0>f'(1)$, }\\
\mbox{$f(u)>0$ in $(0,1)$, $f(u)<0$ in $(1,\infty)$,  $f(u)/u$ is decreasing for $u>0$.}\end{cases}
\]
(The original condition in KPP \cite{kpp} is actually slightly weaker in that it only requires $f'(u)< f'(0)$ for $u> 0$.)

A striking feature of \eqref{cauchy} with $f$ satisfying ${\bf (f_{KPP})}$ is that it predicts consistent successful spreading with an asymptotic spreading speed: There exists $c^*>0$ such that  for any small $\epsilon>0$, 
 \begin{equation}\label{c*}
  \lim_{t\to\infty}U(t,x)=\begin{cases} 1 &\mbox{ uniformly for } x\in [(-c^*+\epsilon)t, (c^*-\epsilon)t],\\
  0 & \mbox{ uniformly for } x\in \mathbb R\setminus [(-c^*-\epsilon)t, (c^*+\epsilon)t].
  \end{cases}
 \end{equation}
 This fact was  proved by Aronson and Weinberger \cite{AW,AW1978}, where $c^*$ was first determined independently by Fisher  and KPP  in 1937, who found in \cite{fisher, kpp}
 that \eqref{cauchy} admits a traveling wave $U(t,x)=\phi(x-ct)$ with $\phi(-\infty)=1$ and $\phi(\infty)=0$  if and only if  $c\geq c^*:=2\sqrt {f'(0)d}$, and they further claimed that $c^*$ is the spreading speed of the species. The existence of a spreading speed was supported by numerous observations of real world examples, the most well-known being the spreading of muskrats in Europe in the early 20th century; see
  \cite{SK, S51} for more details.

In \eqref{cauchy}, the species' dispersal is governed by the diffusion term $d U_{xx}$, known as  ``local diffusion", which does not capture nonlocal factors such as long-range dispersal. To address this shortcoming of \eqref{cauchy}, a number of nonlocal diffusion operators have been proposed and used to replace the local diffusion term $dU_{xx}$. A widely used nonlocal version of \eqref{cauchy} is given by
\begin{equation}
	\label{cau}
	\left\{
	\begin{array}{ll}
		\displaystyle u_t = d \big[\int_{\mathbb{R}} J(x-y) u(t,y) \, dy -  u(t,x)\big] + f(u), & t > 0, \; x \in \mathbb{R}, \\
		u(0, x) = u_0(x), & x \in \mathbb{R},
	\end{array}
	\right.
\end{equation}
where the kernel function $J:\mathbb{R}\mapsto\mathbb{R}$ is assumed to satisfy the following standard condition
\begin{itemize}
\item [${\bf (J):}$] $J\in C(\mathbb{R})\cap L^{\infty}(\mathbb{R})$, $J\geq 0$, $J(0)>0$ and $\displaystyle \int_{\mathbb{R}}J(x)\,dx=1$.
\end{itemize}
Biologically $J(x-y)$ represents the probability that an individual at location $y$ moves to location $x$ in a unit of time, and $d$ represents the frequency of such movements in one time unit. It is well known that \eqref{cauchy} can be viewed as a limiting problem of \eqref{cau} when $J$ is replaced by a sequence of suitable constant multiples of some symmetric kernel functions $J_n$ whose supporting set shrinks to 0 as $n\to\infty$ (see, e.g., \cite{andreu2010, SX}).

 In \eqref{cau}, similar to \eqref{cauchy}, we  assume that $u_0(x)$ is a nonnegative continuous function with nonempty compact support.

When the kernel function $J(x)$ is symmetric (i.e., $J(x)=J(-x)$), \eqref{cau} has been extensively studied, and we refer to \cite{BZ, BCV-JMB, CD, CD07, HMMV, R}
and the references therein for some of the literature. In particular, if $J$ is symmetric  and satisfies the so-called ``thin-tail" condition
\begin{equation}\label{thin-tail}
\int_{\mathbb R} J(x) e^{\mu |x|}dx<\infty \mbox{ for some constant } \mu>0,
\end{equation}
then it follows from \cite{wein1982, yagisita} that \eqref{cau} also exhibits a propagation speed as described by \eqref{c*}, except that now the spreading speed $c^*$ is determined by  the associated traveling wave solution of \eqref{cau}.

In contrast, when $J(x)$ is not symmetric, the understanding of  \eqref{cau} is much less complete. Under the condition \eqref{thin-tail} and several other technical assumptions on $J(x)$, it was shown in \cite{LPL} that the unique solution $u(t,x)$ of \eqref{cau} satisfies
\[
\lim_{t\to\infty} u(t, x+ct)=\begin{cases} 1 & \mbox{ locally uniformly in $x\in\mathbb R$ if } c\in (c_*^-, c_*^+),\\
0 & \mbox{ locally uniformly in $x\in\mathbb R$ if } c\in \mathbb R\setminus [c_*^-, c_*^+],
\end{cases}
\]
where	
\begin{equation}\label{c*+-}
	\begin{cases}
		\displaystyle c_*^- := \sup_{\nu < 0} \frac{d\int_{\mathbb{R}} J(x)e^{\nu x}\,dx - d + f'(0)}{\nu}, \\
		\displaystyle c_*^+ := \inf_{\nu > 0} \frac{d\int_{\mathbb{R}} J(x)e^{\nu x}\,dx - d + f'(0)}{\nu}.
	\end{cases}
	\end{equation}
	This indicates that $c_*^-$ is the leftward spreading speed and $c_*^+$ is the rightward spreading speed. The formulas for $c_*^-$ and $c_*^+$  in \eqref{c*+-} can be found in \cite{MK} as the slowest speeds of the leftward and rightward traveling wave solutions, which can be obtained by looking for traveling waves of the form $u(t,x)=e^{\nu (x-ct)}$ for the associated linear problem
	\[
	\displaystyle u_t = d \big[\int_{\mathbb{R}} J(x-y) u(t,y) \, dy -  u(t,x)\big] + f'(0)u.
	\]
	
	More recently, in \cite{xu}, without the extra technical conditions on $J(x)$ of \cite{LPL}, it was proved,   under the condition \eqref{thin-tail} only, that  $c_*^-<c_*^+$ always holds  and for any small $\epsilon>0$,
	there exists $\delta=\delta_\epsilon>0$ such that 
\begin{equation}\label{xu-li-ruan}
\begin{cases} \lim_{t\to\infty} u(t,x)=0 &\mbox{ uniformly for }  x\in \mathbb R\setminus [(c_*^--\epsilon)t, (c_*^++\epsilon)t],\\
	\liminf_{t \to \infty} u(t,x)  \geq \delta &\mbox{ uniformly for }  x\in [(c_*^-+\epsilon)t, (c_*^+-\epsilon)t].
	\end{cases}
\end{equation}
This result was proved  by constructing suitable upper and lower solutions. One expects that, in the current situation, a more precise result of the form \eqref{c*} holds for $u(t,x)$ with 
$-c^*$ and $ c^*$ there replaced by $c_*^-$ and $ c_*^+$, respectively.  In this paper, we will show that this is indeed the case (as a consequence of a more general result without requiring \eqref{thin-tail}), by using a new approach.

Without assuming the thin-tail condition \eqref{thin-tail},	clearly $c_*^-$ and $c_*^+$ need not be finite. Consider the following left and right thin-tail conditions		
\[\begin{aligned}
	&\mbox{${\bf  (J_{thin}^+):}$ \ \ \ \ \ \ $\displaystyle \int_{0}^{\infty}J(x)e^{\lambda x}\,dx<\infty$ for some $\lambda>0$;}\\
	&\mbox{${\bf  (J_{thin}^-):}$\ \ \ \ \ \ \  $\displaystyle \int^{0}_{-\infty}J(x)e^{-\lambda x}\,dx<\infty$ for some $\lambda>0$.}
\end{aligned}
\]
It is easy to see that $c_*^-$ is finite if and only if ${\bf (J_{thin}^-)}$ holds, and $c_*^+$ is finite if and only if ${\bf (J_{thin}^+)}$ holds. From now on, we define
\begin{equation}\label{c*-gen}
\begin{cases}
c_*^-=-\infty &\mbox{  if ${\bf (J_{thin}^-)}$ does not hold,}\\
c_*^+=\infty &\mbox{  if ${\bf (J_{thin}^+)}$ does not hold.}
\end{cases}
\end{equation}

We will prove the following result on the propagation dynamics of \eqref{cau}:

\begin{theorem}\label{th1.3a} Suppose that ${\bf (J)}$ and ${\bf (f_{KPP})}$ hold, and $u(t,x)$ is the solution to \eqref{cau}. Then
	\begin{align*}
	 \lim_{t\to \infty} u(t,x)=\begin{cases} 1 \mbox{ uniformly for } x\in [a_1t, b_1t] \mbox{ provided that } [a_1,b_1]\subset (c_*^-, c_*^+),\\
	 0 \mbox{ uniformly for $x\leq a_2t$ provided that $c_*^->-\infty$ and $a_2<c_*^-$},\\
	 0 \mbox{ uniformly for $x\geq b_2t$ provided that $c_*^+<\infty$ and $b_2>c_*^+$}.
	 \end{cases}
	\end{align*}
\end{theorem}
Clearly this result implies that the leftward spreading speed $c_*^-=-\infty$ if ${\bf (J_{thin}^-)}$ does not hold,  and the rightward spreading speed $c_*^+=\infty$ if ${\bf (J_{thin}^+)}$ does not hold. In fact, with the above extended definitions of  $c_*^-$ and $ c_*^+$, under condition ${\bf (J)}$ alone, we can always conclude from Theorem \ref{th1.3a} that the leftward spreading speed is $c_*^-$ and the rightward spreading speed is $c_*^+$.

The proof of Theorem \ref{th1.3a} will be based on our results about the asymptotic limit of the principal eigenvalue of the following associated nonlocal operator
\begin{align*}
	\mathcal{L}^c_{\Omega}[\phi](x) := d\int_{\Omega}J(x-y)\phi(y)\,dy - d\phi(x) + c\phi'(x) + f'(0)\phi(x),\quad \phi \in C^1(\Omega)\cap C(\overline{\Omega}),
\end{align*}
where $\Omega=(l_1,l_2)$ is a bounded interval.  By \cite{coville20} and \cite{li}, the principal eigenvalue $\lambda_p(\mathcal{L}^c_{\Omega}) $\footnote{It is easy to see from its definition that $\lambda_p(\mathcal{L}^c_{(a, a+l)})$ is independent of $a$. Moreover, since the eigenvalue problem over $(-l, l)$ only involves the value of $J(x)$ for 
$x\in (-2l, 2l)$, the requirement in \cite{coville20} that $J$ has compact support can always be regarded as satisfied.} of $\mathcal{L}^c_{\Omega}$ always exists.

The key result of this paper is the following theorem, which provides an explicit expression in terms of $J(x)$ for the asymptotic limit of $\lambda_p(\mathcal{L}^c_{(-l,l)})$ as $l\to\infty$.

\begin{theorem}\label{th1.2a}
	Assume that the kernel $J$ satisfies ${\bf (J)}$. Then for any $c\in\mathbb R$,
	\begin{equation}\label{l-to-infty}
	\tilde \lambda_\infty^c:=\lim\limits_{l\to \infty}\lambda_p(\mathcal{L}^c_{(-l,l)}) = \inf\limits_{\nu\in\mathbb{R}}\left\{ d\int_{\mathbb{R}} J(x)e^{-\nu x}\,dx - d + c\nu + f'(0) \right\}.
	\end{equation}
	\end{theorem}
	
	Actually it is easy to show (see Lemma \ref{uni}) that $\inf\limits_{\nu\in\mathbb{R}}\left\{ d\int_{\mathbb{R}} J(x)e^{-\nu x}\,dx - d + c\nu + f'(0) \right\}$ is always achieved by some $\nu=\nu_0\in\mathbb R$. 
	
	The link between Theorem \ref{th1.2a} and the propagation dynamics of \eqref{cau} is provided by the following  relationship between $\lambda_\infty$ and $c_*^-, c_*^+$ defined in \eqref{c*+-} (see Proposition \ref{l9.2}):
\[ 
\tilde\lambda_\infty^c\begin{cases}>0& \mbox{ if  $c\in (c_*^-,c_*^+)$},\\
		<0 & \mbox{ if } c\not\in [c_*^-, c_*^+],\\
		=0 & \mbox{ if } c\in \{c_*^-, c_*^+\},
		\end{cases}
		\]
		which follows easily from  \eqref{l-to-infty}.

	To see how the nonlocal operator $\mathcal{L}^c_{(-l,l)}$ arises from \eqref{cau}, let us note that
	if $u$ solves \eqref{cau}, then for any fixed  $c\in \mathbb R$,
		 $V(t,x):=u(t,x+ct)$ satisfies 
		\begin{equation*}
			\left\{
			\begin{array}{ll}
				\displaystyle V_t=d\int_{\mathbb{R}}J(x-y)V(t,y)\,dy-dV+cV_x+f(V), & t>0, \ x\in\mathbb{R}, \\
				V(0,x)=u_0(x),  & x\in\mathbb{R}.
			\end{array}
			\right.
		\end{equation*}
		
		In order to understand the long-time dynamics of $V(t,x)$, one wishes to first understand the corresponding problem over the bounded interval $(-l, l)$, namely, for fixed $c\in\mathbb R$ and  $l>0$,  the problem
		\begin{equation}\label{l-0}
			\left\{
			\begin{array}{ll}
				\displaystyle V_t=d\int_{-l}^lJ(x-y)V(t,y)\,dy-dV+cV_x+f(V), & t>0,\  x\in (-l, l), \\
				V(t,l)=0 \mbox{\ \ \ \ for } t>0, & \mbox{when } c>0,\\
				V(t, -l)=0 \mbox{\, for } t>0, &\mbox{when } c<0,\\
				V(0,x)=V_0(x),  &x\in [-l, l].
			\end{array}
			\right.
		\end{equation}
		Let us note that, in \eqref{l-0}, the boundary condition is assumed at $x=l$ when $c>0$, and it is assumed at $x=-l$ when $c<0$, and no boundary condition is assumed when $c=0$. These might appear strange on first sight, but if one looks at the non-standard proof for the existence of solutions to \eqref{l-0}, they all become rather natural (the non-standard proof follows an approach first used in \cite{cao}).
		
		Since $V_x$ appears in \eqref{l-0}, some extra conditions on the initial function $V_0$ are needed to guarantee that the solution $V(t,x)$ is differentiable in $x$.
		It suffices to assume 
		\[
		V_0\in \mathbb V^c_l,
		\]
		 where $\mathbb V^c_l$ consists of nonnegative functions in $C([-l,l])$ with the following extra properties when $c\not=0$:
		\[ \phi\in C^1([-l,l]), \
 \begin{cases} \phi(l)=0 \mbox{ and } \ \phi'(l)=\displaystyle -\frac{d}c\int_{-l}^lJ(l-y)\phi(y)dy &\mbox{ if } c>0,\\
 \phi(-l)=0 \mbox{ and } \ \phi'(-l)=\displaystyle\frac dc \int_{-l}^l J(-l+y)\phi(y)dy &\mbox{ if } c<0.
 \end{cases}
\]
It can be shown (see Lemma \ref{lem-V_0}) that $\mathbb V_l^c$ is a dense subset of
\[
\mathbb U_l^c:=\{\phi\in C([-l,l]): \phi\geq 0,\; \phi(l)=0 \mbox{ if } c>0,\ \phi(-l)=0 \mbox{ if } c<0\}.
\]

The following result gives a complete description of the long-time dynamics of \eqref{l-0}, which is parallel to the classical result for logistic equations.
		
		\begin{theorem}\label{c-l}
		Suppose that  ${\bf (J)}$ and ${\bf (f_{KPP})}$ hold, and $V_0\in \mathbb V^c_l$  is  not identically $0$. Then \eqref{l-0} has a unique solution $V(t,x)$ and
		\[
		\lim_{t\to\infty} V(t,x)=\begin{cases} 0 &\mbox{ uniformly in $x\in[-l,l]$ if } \lambda_p(\mathcal L^c_{(-l,l)})\leq 0,\\
		V_l(x) &\mbox{ uniformly in $x\in[-l,l]$ if } \lambda_p(\mathcal L^c_{(-l,l)})> 0,
		\end{cases}
		\]
		where $V_l(x)$ is the unique positive stationary solution of \eqref{l-0}. Moreover, when $\tilde\lambda_\infty^c>0$ and hence
		$\lambda_p(\mathcal L^c_{(-l,l)})> 0$ for all large $l>0$, we have
		\[
		\lim_{l\to\infty} V_l(x)=1 \mbox{ uniformly for $x$ in any bounded interval of $\mathbb R$}.
		\]
		\end{theorem}

Theorem \ref{th1.3a} will follow from an application of Theorem \ref{c-l} and some other  arguments. We would like to mention that traditionally the spreading speed is determined by using traveling wave solutions (as in \cite{AW, AW1978} for \eqref{cauchy}), or by a dynamical systems approach (first introduced in \cite{wein1982} and further developed by  \cite{LZ} and others). In contrast, our proof of Theorem \ref{th1.3a} relies mainly on the properties of the principal eigenvalue $\lambda_p(\mathcal L^c_{(-l,l)})$ and their implications for the dynamics of \eqref{l-0}. 

Theorem \ref{th1.2a} will also find  applications in the understanding of the propagation dynamics of the associated free boundary problem of \eqref{cau}, which will be considered in a separate work. In another separate work, we will extend the results here to the corresponding problems in high space dimensions.

We would like to draw the reader's attention to \cite{BHN, BN} and the references therein for works on asymptotic spreading  of very general local diffusion models in heterogeneous media, where the principal eigenvalues also play a crucial role for the propagation dynamics. In \cite{DJ}, the  spreading speed of a related asymmetric nonlocal diffusion model in time-dependent media was studied.

If we define 
\begin{align*}
\mathcal{L}^{J, c}_{\Omega}[\phi](x):=\int_{\Omega}J(x-y)\phi(y)\,dy+c\phi'(x),
\end{align*}
then clearly
\[
\mathcal L^c_{\Omega}[\phi]=d\,\mathcal L^{J,c/d}_\Omega[\phi]+[f'(0)-d]\phi
\]
and
\[
\lambda_p(\mathcal L^c_{\Omega})=d\,\lambda_p(\mathcal L^{J,c/d}_\Omega)+f'(0)-d.
\]
Therefore \eqref{l-to-infty} is equivalent to
\begin{equation}\label{l-to-infty-c}
\lim\limits_{l\to \infty}\lambda_p(\mathcal{L}^{J, c}_{(-l,l)})=\inf\limits_{\nu\in\mathbb{R}}\big[\int_{\mathbb{R}}J(x)e^{-\nu x}\,dx+c\nu\big].
\end{equation}

 In Section 2, the quantity $\displaystyle \inf\limits_{\nu\in\mathbb{R}}\big[\int_{\mathbb{R}}J(x)e^{-\nu x}\,dx+c\nu\big]$ will be examined and shown to be achieved at some $\nu=\nu_0$ (depending on both $J$ and $c$). To prove \eqref{l-to-infty-c}, it turns out that the cases $c=0$ and $c\not=0$ have to be considered separately. The case $c=0$ will be treated in Section 3, and the case $c\not=0$ in Section 4. Our proof of Theorem \ref{th1.2a} in Sections 3 and 4 rely on results in  \cite{coville20, li} and the construction of a special test function used in a sup-inf characterization of the principal eigenvalue.  Section 5 is devoted to the proof of Theorems \ref{c-l} and \ref{th1.3a}.

\section{On  the quantity $\displaystyle \inf_{\nu \in \mathbb{R}} \big[\int_{\mathbb{R}} J(x)e^{-\nu x}\,dx+c\nu\big]$}

 We give some basic properties of  $\displaystyle \inf_{\nu \in \mathbb{R}} \big[\int_{\mathbb{R}} J(x)e^{-\nu x}\,dx+c\nu\big]$ here, which will be used later in the paper.
\begin{lemma}\label{uni}
	Let \( J \) satisfy condition \({\bf (J)}\). Then the following conclusions hold:
	\begin{enumerate}
		\item[{\rm (1)}]  
		There exists a sequence of compactly supported nonnegative continuous functions \(\{J_n\}\)  such that for every \(n \geq 1\) and \(x \in \mathbb{R}\),
		\[
		J_n(0)>0,\  J_n(x) \leq J_{n+1}(x) \leq J(x), \quad \lim_{n \to \infty} J_n(x) = J(x) \quad \text{locally uniformly in } \mathbb{R}.
		\]
       \item[{\rm (2)}] 
		For any $c\in\mathbb R$, there exists a unique \(\nu_0=\nu_0^c\in\mathbb R\) such that
		\[
		\int_{\mathbb{R}} J(x)e^{-\nu_0 x}\,dx +c\nu_0= \inf_{\nu \in \mathbb{R}} \Big[\int_{\mathbb{R}} J(x)e^{-\nu x}\,dx+c\nu\Big]. 
		\]

	\item[{\rm (3)}] If $J_n(x)$ is given in {\rm (1)} and 
	\[
	\int_{\mathbb{R}}J_n(x)e^{-\nu_n x}\,dx+c\nu_n=\inf\limits_{\nu\in \mathbb{R}}\Big[\int_{\mathbb{R}}J_n(x)e^{-\nu x}\,dx+c\nu\Big],
	\]
	 then,   as $n\to\infty$,
	 	 \[
	  \nu_n\to \nu_0 \mbox{ and } \int_{\mathbb{R}}J_n(x)e^{-\nu_n x}\,dx+c\nu_n\to \inf_{\nu \in \mathbb{R}} \Big[\int_{\mathbb{R}} J(x)e^{-\nu x}\,dx+c\nu\Big].
	 \]
	\end{enumerate}
\end{lemma}

\begin{proof}
(1) Choose a smooth function $\xi_0(x)$ such that
\[
0\leq \xi_0(x)\leq 1,\ \xi_0(x)=1 \mbox{ for } x\in [-1,1],\ \xi_0(x)=0 \mbox{ for } |x|>2.
\]
	Then  $J_n(x):=J(x)\xi_0(x/n)$ has the desired properties.
	
	(2) Since each $J_n$ has compact support, it is easily seen that $\eta_n(\nu):=\int_{\mathbb{R}}J_n(x)e^{-\nu x}\,dx+c\nu$ is a smooth function of $\nu$ over $\mathbb R$, with
	\[
	\eta_n''(\nu)=\int_{\mathbb{R}} J_n(x)x^2 e^{-\nu x}\,dx> 0 \mbox{ for all  $\nu\in\mathbb R$.}
	\]
	Since $J_n(0)>0$ there exists $\delta_n\in (0, 1/2)$ such that $J_n(x)\geq  \sigma_n>0$ for $x\in [-2\delta_n,2\delta_n]$. Therefore, 
	\[
	\eta_n(\nu)\geq  \int_{-2\delta_n}^{-\delta_n}J_n(x)e^{-\nu x}dx+  \int_{\delta_n}^{2\delta_n}J_n(x)e^{-\nu x}dx+c\nu>\sigma_n\delta_n e^{\delta_n|\nu|}+c\nu\to \infty \mbox{ as } |\nu|\to \infty.
	\]
	It follows that there exists a unique  $\nu_n$ such that 
	\begin{equation}\label{nu_n}
	\eta_n(\nu_n)=\inf_{\nu\in \mathbb{R}}\eta_n(\nu).
	\end{equation}

We claim that \(\{ \nu_n \}\) is a bounded sequence. Otherwise, by passing to a subsequence we may assume that $|\nu_n|\to\infty$ as $n\to\infty$.

	Since $J(0)>0$ there exists $\delta\in (0, 1/2)$ such that $J_n(x)\geq \frac 12 J(x)\geq \sigma_0>0$ for $x\in [-2\delta,2\delta]$ and all large $n$. Therefore, for such $n$,
	\begin{equation}\label{eta''}
	\eta_n(\nu)\geq  \int_{-2\delta}^{-\delta}J_n(x)e^{-\nu x}dx+  \int_{\delta}^{2\delta}J_n(x)e^{-\nu x}dx+c\nu>\sigma_0\delta e^{\delta|\nu|}+c\nu,
	\end{equation}
	 and then
\[
\eta_n(\nu_n)\geq  \sigma_0\delta e^{\delta|\nu_n|}+c\nu_n\to\infty
\mbox{ as } n\to\infty,
\]
which is a contradiction to
\[
\eta_n(\nu_n)\leq \eta_n(0)\leq 1.
\]
Therefore there exists a subsequence, still denoted by \(\nu_n\), such that   $\nu_n\to \nu_0$ as $n\to\infty$.
	By Fatou's lemma we obtain
	$$ \int_{\mathbb{R}}J(x)e^{-\nu_0 x}\,dx+c\nu_0\leq\liminf\limits_{n\to\infty}\Big[\int_{\mathbb{R}}J_n(x)e^{-\nu_n x}\,dx+c\nu_n\Big]\leq 1+c\nu_0.$$
	On the other hand, 
	\begin{align*}
		\int_{\mathbb{R}}J_n(x)e^{-\nu_n x}\,dx+c\nu_n&= \inf\limits_{\nu\in \mathbb{R}}\Big[\int_{\mathbb{R}}J_n(x)e^{-\nu x}\,dx+c\nu\Big]\\
		&\leq \inf\limits_{\nu\in \mathbb{R}}\Big[\int_{\mathbb{R}}J(x)e^{-\nu x}\,dx+c\nu]\leq \int_{\mathbb{R}}J(x)e^{-\nu_0 x}\,dx+c\nu_0.
	\end{align*}
	Therefore we must have
	\begin{align*}
		\inf\limits_{\nu\in \mathbb{R}}\Big[\int_{\mathbb{R}}J(x)e^{-\nu x}\,dx+c\nu\Big]=\int_{\mathbb{R}}J(x)e^{-\nu_0 x}\,dx+c\nu_0.
	\end{align*}
	The existence of $\nu_0$ is now proved.

We  next show the uniqueness of $\nu_0$ satisfying 
 \( \eta(\nu_0) 
 = \inf_{\nu \in \mathbb{R}} \eta(\nu) \). If there exists \( \nu_* \neq \nu_0 \) such that \( \eta(\nu_*) = \eta(\nu_0) \), then we are going to deduce a contradiction.  Let $\nu_n$ be given by \eqref{nu_n}  with $\nu_n\to \nu_0$  as $n\to\infty$.   Similar to \eqref{eta''} there exists $m> 1$ such that
 \begin{equation}\label{eta_n''}
 \eta_n''(\nu)\geq  \int_{-2\delta}^{-\delta}J_n(x)x^2e^{-\nu x}dx+  \int_{\delta}^{2\delta}J_n(x)x^2e^{-\nu x}dx>\sigma_0\delta^3 e^{\delta|\nu|}\geq  \sigma_0\delta^3 \mbox{ for all } n\geq m,\ \nu\in\mathbb R.
 \end{equation}
 So in the case \( \nu_* > \nu_0 \),
it follows that $\eta_n'(\nu_n)=0$, $\eta_n'(\nu)\geq \sigma_0\delta^3 (\nu-\nu_n)$ for $\nu>\nu_n$ and $n\geq m$. Therefore, for all large $n$, 
 \[
 \eta_n(\nu_*)-\eta_n(\nu_n)\geq \sigma_0\delta^3(\nu_*-\nu_n)^2/2.
 \]
 Letting $n\to\infty$ we deduce 
 \[
 0=\eta(\nu_*)-\eta(\nu_0)\geq \sigma_0\delta^3(\nu_*-\nu_0)^2/2 >0.
 \]
 The case $\nu^*<\nu_0$ similarly leads to a contradiction.
 This  proves the uniqueness of $\nu_0$, and
   the proof for (2) is complete.
	
(3) The conclusions follow directly from the uniqueness of $\nu_0$ and the argument in (2).
\end{proof}

\begin{corollary}\label{coro1} Suppose that $J$ satisfies ${\bf (J)}$.  Then the following conclusions hold:

\begin{itemize}
\item[{\rm (1)}]  If neither \({\bf (J_{thin}^+)}\) nor \({\bf (J_{thin}^-)}\) is satisfied, then  
\begin{align}\label{9}
 \lambda^c_{\infty}:=\inf\limits_{\nu\in \mathbb{R}}\Big[\int_{\mathbb{R}}J(x)e^{-\nu x}\,dx+c\nu\Big]=\int_{\mathbb R} J(x)dx
= 1.
\end{align}
\item[{\rm (2)}] If  $\int_{\mathbb{R}} J(x) x \,dx=0$ in the sense that for some sequence $\{J_n\}$ satisfying (1) of Lemma \ref{uni},
\[
\lim_{n\to\infty}\int_{\mathbb R} J_n(x)xdx=0,
\]
then  \eqref{9} holds for $c=0$.
\end{itemize}
\end{corollary}
\begin{proof} (1) Since neither \({\bf (J_{thin}^+)}\) nor \({\bf (J_{thin}^-)}\) is satisfied, 
$\int_{\mathbb{R}}J(x)e^{-\nu x}\,dx=\infty$ for $\nu\not=0$. Hence \eqref{9} holds.

(2) With $J_n$  satisfying (1) of Lemma \ref{uni} and
$
\lim_{n\to\infty}\int_{\mathbb R} J_n(x)xdx=0,
$ we see that for every $n\geq 1$,
 $G_n(\nu):=\int_{\mathbb{R}}J_n(x)e^{-\nu x}\,dx$ is a smooth function over $\mathbb R$, and similar to \eqref{eta_n''},  there exists $\delta\in (0, 1/2)$, $\sigma>0$ and integer $m>1$ such that
 \[
 G_n''(\nu)=\int_{\mathbb{R}}J_n(x)x^2e^{-\nu x}\,dx\geq \sigma e^{\delta|\nu|}\geq \sigma \mbox{ for all } \nu\in\mathbb R,\ n\geq m.
 \]
 By assumption,
 \[
 \epsilon_n:= G_n'(0)=\int_{\mathbb{R}} J_n(x) x \,dx\to 0 \mbox{ as } n\to\infty.
 \]
 Hence we have, for $n\geq m$,
 \begin{equation}\label{Gn'}
 |G_n'(\nu)-\epsilon_n|=|\int_0^\nu G_n''(\xi)d\xi |\geq \sigma |\nu|  \mbox{ for all } \nu\in\mathbb R.
  \end{equation}
By Lemma \ref{uni}, there exists a unique $\nu_n$ such that 
\[
G_n(\nu_n)=\min_{\nu\in\mathbb R} G_n(\nu),\ G_n'(\nu_n)=0,\ \nu_n\to \nu_0 \mbox{ as } n\to\infty,
\]
and
\[
\inf\limits_{\nu\in \mathbb{R}}\int_{\mathbb{R}}J(x)e^{-\nu x}\,dx=\int_{\mathbb{R}}J(x)e^{-\nu_0 x}\,dx.
\]
 
  Therefore to complete the proof it suffices to show $\nu_0= 0$. Taking $\nu=\nu_n$ in  \eqref{Gn'} we obtain
  \[
   \sigma|\nu_n|\leq \epsilon_n\to 0 \mbox{ as } n\to\infty.
     \]
  Therefore $\nu_0=0$ and the proof is complete.
  \end{proof}
  
  For any symmetric kernel function $J$, we can find $J_n$ symmetric satisfying (1) of Lemma \ref{uni}, and hence $\int_{\mathbb{R}} J(x) x \,dx=0$ and we can apply the above result to conclude that \eqref{9} holds for $c=0$.

\section{On the principal eigenvalue of $\mathcal L^{J, 0}_\Omega$}\label{subsec}
To simplify notations, we will denote  $\mathcal L^{J, 0}_\Omega$ by $\mathcal{L}_{\Omega}^J$, where $\Omega$ is a bounded open interval in $\mathbb{R}$  and $J$ satisfies ${\bf (J)}$. Denote
\begin{align}\label{2.2c}
\lambda_p(\mathcal{L}_{\Omega}^J):=\inf\left\{\lambda\in \mathbb{R}:\mathcal{L}_{\Omega}^J[\phi]\leq \lambda\phi\,\,  \text{in $\Omega$ for some}\, \phi\in C(\overline{\Omega}),\,\phi>0\right\}.
\end{align}
Then by \cite{li},   $\lambda_p(\mathcal{L}_{\Omega}^J)$\footnote{Note that
$-\lambda_p(\mathcal{L}_{\Omega}^J):=\sup\left\{\mu\in \mathbb{R}:\mathcal{L}_{\Omega}^J[\phi] +\mu \phi\leq 0\,\,  \text{for some}\, \phi\in C(\overline{\Omega}),\phi>0\right\}.$} is an eigenvalue of the operator $\mathcal{L}_{\Omega}^J$ with a  positive eigenfunction in $C(\bar\Omega)$, and $\lambda_p(\mathcal{L}_{\Omega}^J)$ is known as the principal eigenvalue of $\mathcal{L}_\Omega^J$. 

In the following, we focus on the case of $\Omega=(a, a+l)$ with $a\in\mathbb R$ and $l\in (0,\infty)$.

\begin{lemma}\label{l2.1}
	Assume that the kernel $J$ satisfies ${\bf (J)}$. Then the following statements hold true:
	\begin{itemize}
		\item [{\rm (i)}] $\lambda_p(\mathcal{L}^J_{(a, a+l)})$ is independent of $a$ and strictly increasing and continuous in $l$;
		\item [{\rm (ii)}] $\lim\limits_{l\to 0}\lambda_p(\mathcal{L}^J_{(a, a+l)})=0$.
	\end{itemize}
\end{lemma}

\begin{proof}
	It is easily checked from the definition  that $\lambda_p(\mathcal{L}^J_{(a, a+l)})$ is independent of $a$. So it suffices to prove the stated conclusions for $\lambda_p(\mathcal{L}^J_{(0, l)})$.
	
	The continuity of the eigenvalue and conclusion (ii) can be established using a similar approach to Proposition 3.4-(i),(iii) in \cite{cao}, and the proof is therefore omitted. Hence, we focus on proving the monotonicity properties of the eigenvalue with respect to the length of the domain.

Suppose $\hat{l}>l$. To simplify  notations, we denote  $\lambda_l:=\lambda_p(\mathcal{L}^J_{(0,l)})$ and $\lambda_{\hat l}:=\lambda_p(\mathcal{L}^J_{(0,\hat{l})})$. From the definition of $\lambda_p(\mathcal L_\Omega^J)$ we easily see that $\lambda_l\leq \lambda_{\hat l}$. We need to show the strict inequality: $\lambda_l< \lambda_{\hat l}$.
	
	Arguing by contradiction, we assume that $\lambda_l= \lambda_{\hat l}$. Let $(\lambda_l,\phi)$ and $(\lambda_{\hat l},\hat{\phi})$ be the principal eigenpairs of the nonlocal operators $\mathcal{L}^J_{(0,l)}$ and $\mathcal{L}^J_{(0,\hat l)}$, respectively. Then using $\lambda_l=\lambda_{\hat l}$ we obtain
	\begin{equation}\label{ev1}
		\lambda_l\hat{\phi}(x)=\int_0^{\hat{l}}J(x-y)\hat{\phi}(y)\,dy,\quad x\in[0,\hat{l}]
	\end{equation}
	and
	\begin{equation}\label{ev2}
		\lambda_l\phi(x)=\int_0^{l}J(x-y)\phi(y)\,dy,\quad x\in[0,l].
	\end{equation}
	
	Since both $\phi$ and $\hat\phi$ are positive and continuous over $[0, l]$,
	\[
	\delta_0:=\sup\{\delta>0: \hat \phi(x)\geq \delta \phi(x) \mbox{ for } x\in [0, l]\}
	\]
	is a well defined positive number, and $\hat \phi(x)\geq \delta_0\phi(x) $ in $[0, l]$. 
	
	The proof of $\lambda_p< \hat{\lambda}_p$ will be carried out according to the following two cases.
	
	\underline{\bf Case 1}.  $\hat{\phi}(x)\equiv \delta_0 \phi(x)$ in $[0,l]$.
	
	Then by \eqref{ev1} and $J(0)>0$ we deduce
	$$\lambda_l\phi(x)=\int_0^{l}J(x-y)\phi(y)\,dy+\frac{1}{\delta_0}\int_l^{\hat{l}}J(x-y)\hat{\phi}(y)\,dy>\int_0^{l}J(x-y)\phi(y)\,dy,$$
	for $x<l$ sufficiently  close to $l$, which is a contradiction to \eqref{ev2}.
	
	\underline{\bf Case 2}.  $\hat{\phi}(x)\not\equiv \delta_0 \phi(x)$ in $ [0,l]$.
	
	Then $\{x\in [0,l]:\hat{\phi}(x)-\delta_0 \phi(x)=0\}$ is a non-empty proper subset of $[0, l]$ which contains at least one point $x_0\in \partial\{x\in [0,l]:\hat{\phi}(x)-\delta_0 \phi(x)> 0\}$. Using \eqref{ev1}, \eqref{ev2} and $J(0)>0$, we easily obtain, for $v(x):=\hat \phi(x)-\delta_0\phi(x)$,
	$$0=\lambda_l v(x_0)\geq\int_0^{l}J(x_0-y)v(y)\,dy>0,$$
so we also arrive at a  contradiction.   The proof is complete.	
\end{proof}

\begin{theorem}\label{l2.2}
Assume that  $J$ satisfies ${\bf (J)}$. Then
\begin{align*}
\lim\limits_{l\to \infty}\lambda_p(\mathcal{L}^J_{(-l,l)})=\inf\limits_{\nu\in \mathbb{R}}\int_{\mathbb{R}}J(x)e^{-\nu x}\,dx.
\end{align*}
\end{theorem}

\begin{proof}
Denote $\lambda_l:=\lambda_p(\mathcal{L}^J_{(0,2l)})=\lambda_p(\mathcal L^J_{(-l,l)})$. We already know that $\lambda_l$ is continuous and strictly increasing in $l$. 

 {\bf Step 1}. We show that $\lim\limits_{l\to \infty}\lambda_l\leq \displaystyle\inf\limits_{\nu\in \mathbb{R}}\int_{\mathbb{R}}J(x)e^{-\nu x}\,dx\leq 1$.

From the definition of $\lambda_p(\mathcal L_{(-l, l)}^J)$ and the existence of  principal eigenfunctions, we easily obtain 
\begin{align*}\label{6a}
\lambda_l=\inf_{\phi\in C([-l,l]),\, \phi>0\,\text{on}\, [-l,l]}\ \sup_{x\in(-l,l)}\frac{\mathcal{L}^J_{(-l,l)}[\phi](x)}{\phi(x)}.
\end{align*}
 Hence, for any $\nu\in \mathbb{R}$,
$$\lambda_l\leq \sup_{x\in(-l,l)}\frac{\int_{-l}^lJ(x-y)e^{\nu y}\,dy}{e^{\nu x}}=\sup_{x\in(-l,l)}\int_{x-l}^{x+l}J(y)e^{-\nu y}\,dy\leq \int_{\mathbb{R}}J(y)e^{-\nu y}\,dy,$$
which clearly implies the desired inequality.
\medskip

{\bf Step 2}. We prove  $\lim\limits_{l\to \infty}\lambda_l\geq  \inf\limits_{\nu\in \mathbb{R}}\int_{\mathbb{R}}J(x)e^{-\nu x}\,dx$ when $J$ is compactly support.

By Lemma \ref{uni} there exists $\nu_0\in \mathbb{R}$ such that
 \begin{equation}\label{infi}
  \lambda_{\infty}:=\inf\limits_{\nu\in \mathbb{R}}\int_{\mathbb{R}}J(x)e^{-\nu x}\,dx=\int_{\mathbb{R}}J(x)e^{-\nu_0 x}\,dx.
\end{equation}

Fix   some constant $M>0$ such that $\text{supp}{\bf (J)}\subset (-M,M)$. Since $J(0)>0$, we can find a large positive number $\sigma$ such that 
\begin{equation}\label{largeen}
\mu_\sigma:=\min\Big\{\int_0^M J(y)e^{-(\nu_0-\sigma)y}\,dy,\ \int_{-M}^0 J(y)e^{-(\nu_0+\sigma)y}\,dy\Big\}>1\geq \lambda_{\infty}.
\end{equation}

For any given $\epsilon>0$, there exists a large positive integer $k$ (depending on $\epsilon$) such that
\begin{equation}\label{les}
0<(1-e^{-\frac\sigma k M})\mu_\sigma <\epsilon.
\end{equation}

For large $l$ satisfying $M<\frac 14 \frac{l}{k+1}$, and $j=0, 1,2,..., k$, define
\[
\delta_0:=\frac\sigma k,\ I_j^+:=\Big[\frac{j}{k+1}l,\ \frac{j+1}{k+1}l\Big],\ I_j^-:=\Big[-\frac{j+1}{k+1}l,\ -\frac{j}{k+1}l\Big],
\]
and
\begin{equation*}
\varphi(x):=\left\{
\begin{array}{ll}
\displaystyle    e^{(\nu_0-j\delta_0)x+\delta_0\Sigma_{i=1}^j \frac i{k+1}l}=e^{\nu_0 x+j\delta_0(\frac{j+1}2 \frac{l}{k+1}-x)}, & x\in I^+_j,\ j=1,..., k,\medskip\\
    e^{\nu_0 x}, & x\in I_0^+\cup I_0^-, \\
   \displaystyle  e^{(\nu_0+j\delta_0)x+\delta_0\Sigma_{i=1}^j \frac i{k+1}l}=e^{\nu_0 x+j\delta_0(\frac{j+1}2 \frac{l}{k+1}+x)}, & x\in I^-_j,\ j=1,..., k.
  \end{array}
\right.
\end{equation*}
Obviously, $\varphi$ is  continuous  and piecewise smooth on the interval $[-l,l]$. Using Theorem 2.4 in \cite{li} with $k(x,y)=-J(x-y)$ we obtain
\begin{align}\label{6a}
\lambda_l=\sup_{\phi\in C([-l,l]),\, \phi>0\,\text{on}\, [-l,l]}\ \inf_{x\in(-l,l)}\frac{\mathcal{L}^J_{(-l,l)}[\phi](x)}{\phi(x)}.
\end{align}
Therefore
\[
 \lambda_l\geq \inf\limits_{x\in(-l,l)}e_l(x) \mbox{ with }
e_l(x):=\frac{\mathcal{L}^J_{(-l,l)}[\varphi](x)}{\varphi(x)}=\frac{\displaystyle\int_{x-l}^{x+l}J(y)\varphi(x-y)\,dy}{\varphi(x)}.
\]
Hence, to show the desired estimate   it suffices to show that for any $\epsilon>0$, there exists large $l>0$ such that
\begin{align}\label{belowesti}
\inf\limits_{x\in(-l,l)}e_l(x)>\lambda_{\infty}-\epsilon.
\end{align}

Next we estimate $e_l(x)$ according to the location of $x$ in $ [-l, l]$. Firstly we notice that  
\[
[x-l, x+l]\supset [-M, M] \mbox{ for $x\in [-l+M, l-M]$.}
\]

\underline{\bf Case 1}.  $0\leq x\leq \frac{l}{k+1}-M$.

For $x$ in this range we have 
\[
\mbox{$-\frac{l}{k+1}\leq x-M<x+M\leq \frac l{k+1}$.}
\]
 Therefore
$$e_l(x)=\frac{\int_{-M}^{M}J(y)\varphi(x-y)\,dy}{\varphi(x)}=\int_{-M}^{M}J(y)e^{-\nu_0y}\,dy=\lambda_{\infty}.$$

\underline{\bf Case 2}.  $\frac{l}{k+1}-M< x\leq \frac{l}{k+1}$.

For $x$ in this range by  the definition of $\varphi$,  \eqref{infi} and \eqref{les}, we obtain 
\begin{align*}
e_l(x) &=\frac{\int_{x-l}^{x+l}J(y)\varphi(x-y)\,dy}{\varphi(x)}=\frac{\int_{-M}^{M}J(y)\varphi(x-y)\,dy}{\varphi(x)}\\
&=\int_{x-\frac{l}{k+1}}^{M}J(y)e^{-\nu_0y}\,dy+e^{\delta_0(\frac{l}{k+1}-x)}\int_{-M}^{x-\frac{l}{k+1}}J(y)e^{-(\nu_0-\delta_0)y}\,dy\\
&\geq \int_{x-\frac{l}{k+1}}^{M}J(y)e^{-\nu_0y}\,dy+\int_{-M}^{x-\frac{l}{k+1}}J(y)e^{-(\nu_0-\delta_0)y}\,dy\\
&=\int_{-M}^{M}J(y)e^{-\nu_0y}\,dy+\int_{-M}^{x-\frac{l}{k+1}}J(y)e^{-\nu_0y}(e^{\delta_0y}-1)\,dy\\
&\geq \lambda_{\infty}-\int_{-M}^{0}J(y)e^{-\nu_0 y}(1-e^{\delta_0 y})\,dy\\
&\geq\lambda_{\infty}-(1-e^{-\delta_0 M})\int_{-M}^{0}J(y)e^{-\nu_0 y}\,dy\\
&>\lambda_{\infty}-\epsilon,
 \end{align*}
where we have used the fact that  $\int_{-M}^{0}J(y)e^{-\nu_0 y}\,dy<\lambda_{\infty}\leq 1$.

\underline{\bf Case 3}.  $\frac{jl}{k+1}< x\leq \frac{jl}{k+1}+M$, $1\leq j\leq k$.

By the definition of $\varphi$, we obtain
\begin{align*}
	e_l(x) &=\frac{\int_{x-l}^{x+l}J(y)\varphi(x-y)\,dy}{\varphi(x)}=\frac{\int_{-M}^{M}J(y)\varphi(x-y)\,dy}{\varphi(x)}\\
	&=\frac{\int_{x-\frac{jl}{k+1}}^{M}J(y)\varphi \Big|_{I_{j-1}^+}(x-y)\,dy+\int_{-M}^{x-\frac{jl}{k+1}}J(y)\varphi \Big|_{I_{j}^+}(x-y)\,dy}{\varphi\Big|_{I_j^+}(x)}\\
	&=e^{\delta_0(x-\frac{jl}{k+1})}\int_{x-\frac{jl}{k+1}}^{M}J(y)e^{-(\nu_0-j\delta_0+\delta_0)y}\,dy+\int_{-M}^{x-\frac{jl}{k+1}}J(y)e^{-(\nu_0-j\delta_0)y}\,dy\\
	&\geq \int_{x-\frac{jl}{k+1}}^{M}J(y)e^{-(\nu_0-j\delta_0+\delta_0)y}\,dy +\int_{-M}^{x-\frac{jl}{k+1}}J(y)e^{-(\nu_0-j\delta_0)y}\,dy\\
	&=\int_{-M}^{M}J(y)e^{-(\nu_0-j\delta_0)y}\,dy-\int_{x-\frac{jl}{k+1}}^{M}J(y)e^{-(\nu_0-j\delta_0)y}(1-e^{-\delta_0y})\,dy\\
	&\geq \int_{-M}^{M}J(y)e^{-(\nu_0-j\delta_0)y}\,dy-(1-e^{-\delta_0M})\int_{0}^{M}J(y)e^{-(\nu_0-k\delta_0)y}\,dy\\
	&> \lambda_{\infty}-\epsilon \mbox{ by the minimality property of $\nu_0$ and \eqref{les}.}
\end{align*}

\underline{\bf Case 4}.  $\frac{jl}{k+1}+M< x\leq \frac{(j+1)l}{k+1}-M$, $1\leq j\leq k$.

Making use of the definition of $\varphi$ and \eqref{infi}, we have
$$e_l(x)=\frac{\int_{-M}^{M}J(y)\varphi \Big|_{I_j^+}(x-y)\,dy}{\varphi \Big|_{I_j^+}(x)}
=\int_{-M}^{M}J(y)e^{-(\nu_0-j\delta_0)y}\,dy>\lambda_{\infty}.$$

\underline{\bf Case 5}.   $\frac{(j+1)l}{k+1}-M< x\leq \frac{(j+1)l}{k+1}$, $1\leq j\leq k-1$.

Using  the definition of $\varphi$ and \eqref{infi}, \eqref{les}, we deduce
\begin{align*}
e_l(x)
&=\int_{x-\frac{(j+1)l}{k+1}}^{M}J(y)e^{-(\nu_0-j\delta_0)y}\,dy+e^{\delta_0(\frac{(j+1)l}{k+1}-x)}\int_{-M}^{x-\frac{(j+1)l}{k+1}}J(y)e^{-(\nu_0-j\delta_0-\delta_0)y}\,dy\\
&\geq \int_{x-\frac{(j+1)l}{k+1}}^{M}J(y)e^{-(\nu_0-j\delta_0)y}\,dy+\int_{-M}^{x-\frac{(j+1)l}{k+1}}J(y)e^{-(\nu_0-j\delta_0-\delta_0)y}\,dy\\
&=\int_{-M}^{M}J(y)e^{-(\nu_0-j\delta_0)y}\,dy+\int_{-M}^{x-\frac{(j+1)l}{k+1}}J(y)e^{-(\nu_0-j\delta_0)y}(e^{\delta_0y}-1)\,dy\\
&>\lambda_{\infty}-(1-e^{-\delta_0M})\int_{-M}^{0}J(y)e^{-\nu_0y}\,dy\\
&>\lambda_{\infty}-\epsilon,
 \end{align*}
 where we have used again the inequality $\int_{-M}^{0}J(y)e^{-\nu_0 y}\,dy<\lambda_{\infty}\leq 1$.

\underline{\bf Case 6}.   $x\in (l-M,l]$.

 Using \eqref{largeen}, we obtain
\begin{align*}
e_l(x) &=\frac{\int_{x-l}^{M}J(y)\varphi(x-y)\,dy}{\varphi(x)}=\frac{\int_{x-l}^{M}J(y)\varphi\Big|_{I_k^+}(x-y)\,dy}{\varphi\Big|_{I_k^+}(x)}\\ 
&=\int_{x-l}^{M}J(y)e^{-(\nu_0-k\delta_0)y}\,dy\geq\int_{0}^{M}J(y)e^{-(\nu_0-k\delta_0)y}\,dy >\lambda_{\infty}.
 \end{align*}

To sum up, we have proved that $e_l(x)>\lambda_{\infty}-\epsilon$ for $x\in [0,l]$ and all large $l$. Similarly, we can show $e_l(x)>\lambda_{\infty}-\epsilon$ for $x\in [-l,0]$ and all large $l$. Thus we have proved the validity of \eqref{belowesti},
and Step 2 is finished.\medskip

{\bf Step 3}. We prove  $\lim\limits_{l\to \infty}\lambda_l\geq  \inf\limits_{\nu\in \mathbb{R}}\int_{\mathbb{R}}J(x)e^{-\nu x}\,dx$ for a general kernel function $J$ satisfying ${\bf (J)}$, without the  compact support assumption.

Let $J_n$ be given as in Lemma \ref{uni}. Then
by \eqref{6a},
\begin{align*}
	\lambda_l=&\sup\limits_{\phi\in C([-l,l]),\, \phi>0\,\text{on}\, [-l,l]}\inf\limits_{x\in(-l,l)}\frac{\int_{-l}^{l}J(x-y)\phi(y)\,dy}{\phi(x)}\\
	\geq& \sup\limits_{\phi\in C([-l,l]),\, \phi>0\,\text{on}\, [-l,l]}\inf\limits_{x\in(-l,l)}\frac{\int_{-l}^{l}J_n(x-y)\phi(y)\,dy}{\phi(x)}=:\lambda_l^n.
\end{align*}
Since $J_n$ has compact support, by the conclusions in Step 2 above we have
\begin{align*}
	\lambda_\infty^n:=\lim_{l\to\infty}\lambda_l^n=\inf\limits_{\nu\in \mathbb{R}}\int_{\mathbb{R}}J_n(x)e^{-\nu x}\,dx=\int_{\mathbb{R}}J_n(x)e^{-\nu_n x}\,dx.
\end{align*}
It follows that
\begin{align}\label{limi}
		\lim_{l\to\infty}\lambda_l\geq \lambda_\infty^n\ \ \ {\rm for\ all}\ n\geq 1.
\end{align}

By part (3) of Lemma \ref{uni}, $\lambda_\infty^n\to \lambda_\infty$ as $n\to\infty$. Hence \eqref{limi} implies
$$\lim_{l\to\infty}\lambda_l\geq \lambda_\infty= \inf\limits_{\nu\in \mathbb{R}}\int_{\mathbb{R}}J(x)e^{-\nu x}\,dx.$$
This finished Step 3.
The desired conclusion clearly follows directly from Step 1 and Step 3.
The proof is now complete.
\end{proof}
\medskip

For a given function $J$ satisfying ${\bf (J)}$, it is in general difficult to check whether $\inf\limits_{\nu\in \mathbb{R}}\int_{\mathbb{R}}J(x)e^{-\nu x}\,dx<1.$
The following example describes a situation that the asymmetry of $J$  may decrease the value of $\inf\limits_{\nu\in \mathbb{R}}\int_{\mathbb{R}}J(x)e^{-\nu x}\,dx.$

{\bf Example 1.} Let $J_a(x)=\xi_{[a-1,a]}(x)$ be the characteristic function of the interval
 \([a-1, a]\) with \(0.5 \leq a < 1\), and 
 \[
  G_a(\nu):=\int_{\mathbb{R}}J_a(x)e^{-\nu x}\,dx.
 \]
  A simple calculation gives
		\begin{align*}
		\lambda_a:=\inf\limits_{\nu\in \mathbb{R}}G_a(\nu)=\inf_{\nu\in\mathbb{R}}\frac{e^\nu-1}{\nu}e^{-\nu a} ,
	\end{align*}
	and it is easily seen that
both $G_a(\nu)$ and $	\lambda_a=\inf\limits_{\nu\in \mathbb{R}}G_a(\nu) $ are  decreasing  with  respect to $ a$ in $[0.5, 1)$.
	Moreover, 
	\begin{align*}
	\lambda_{0.5}=1, \lim_{a\to 1}\lambda_a=0.
	\end{align*}
Although $J_a$ is not continuous, but if $J$ satisfies ${\bf (J)}$ and is close to $J_a$ then $\int_{\mathbb R} J(x) e^{-\nu x}dx$
behaves similarly to $G_a(\nu)$.
\medskip

{\bf Example 2}.  In the opposite direction to Example 1 above, it is easy to find asymmetric functions $J$ satisfying ${\bf (J)}$ as well as $\int_{\mathbb{R}} J(x) x \,dx=0$.
Then by Corollary \ref{coro1}  we have $\inf_{\nu\in\mathbb R}\int_{\mathbb R} J(x) e^{-\nu x}dx=1$.

\section{On the principal eigenvalue of $\mathcal{L}^{J,c}_\Omega$}\label{sec2.2}

Suppose that $\Omega=(l_1,l_2)$ is a bounded open interval in $\mathbb{R}$,  and  $J$ satisfies ${\bf (J)}$.
For $c\in \mathbb{R}\setminus\{0\}$  we define
 \begin{equation*}
X_{\Omega}:=\begin{cases}
\{\phi\in C^1(\Omega)\cap C(\overline{\Omega}): \phi>0\ \text{in}\ \Omega \}& \text{if}\ \Omega=(l_1,l_2),\\
	\{\phi\in C^1(\mathbb{R})\cap L^\infty (\mathbb{R}): \phi>0\ \text{in}\ \mathbb{R}\}& \text{if}\  \Omega=\mathbb{R},\\
\end{cases}
\end{equation*}
\begin{equation*}
X^c_{\Omega}:=\begin{cases}
\{\phi\in X_\Omega:  \phi(l_2)=0\}& \text{if}\ c>0,\ \Omega=(l_1,l_2),\\
	\{\phi\in X_\Omega:  \phi(l_1)=0\}& \text{if}\ c<0,\ \Omega=(l_1,l_2),\\
\end{cases}
\end{equation*}
and
\begin{align}\label{lambdaL2}
\lambda_p(\mathcal{L}^{J, c}_{\Omega}):=&\inf\{\lambda\in \mathbb{R}:\mathcal{L}^{J, c}_{\Omega}[\phi]\leq \lambda\phi\ \text{in $\Omega$ for some}\  \phi\in X_{\Omega}\}.
\end{align}
Then by \cite[Theorems 1.2 and 1.4]{coville20},   $\lambda_p(\mathcal{L}^{J,c}_{\Omega})$\footnote{To compare with the definition in \cite{coville20}, let us note that \[-\lambda_p(\mathcal{L}^{J, c}_{\Omega})=\sup\{\mu \in \mathbb{R}:\mathcal{L}^{J, c}_{\Omega}[\phi]+ \mu\phi\leq 0\ \text{in $\Omega$ for some}\  \phi\in X_{\Omega}\}.\]} is an eigenvalue of the operator $\mathcal{L}^{J,c}_{\Omega}$ with a continuous and positive eigenfunction, and  it is referred to as the principal eigenvalue of $\mathcal L^{J,c}_\Omega$. So there exists $\phi=\phi^{J,c}_{\Omega}\in X_{\Omega}$  such that 
\[
\mathcal{L}^{J, c}_{\Omega}[\phi](x)=\lambda_p(\mathcal{L}^{J, c}_{\Omega}) \phi(x) \mbox{ for $x\in \Omega$.}
\]
Moreover, when
 $\Omega=(l_1,l_2)$ is a bounded interval, by \cite[Theorem 4.3]{li} and \cite[Theorem 1.2]{coville20},
\begin{align}\label{minmax}
\phi^{J,c}_{\Omega}\in X^c_{\Omega} \mbox{ and } \lambda_p(\mathcal{L}^{J, c}_{\Omega})=\inf\limits_{\phi\in X^c_{\Omega}}\sup\limits_{x\in \Omega}\frac{\mathcal{L}^{J, c}_{\Omega}[\phi]}{\phi}=\sup\limits_{\phi\in X^c_{\Omega}}\inf\limits_{x\in \Omega}\frac{\mathcal{L}^{J, c}_{\Omega}[\phi]}{\phi}.
\end{align}

From the definition of $\lambda_p(\mathcal{L}^{J, c}_{\Omega})$,  Proposition 1.3 and Theorem 1.4 in \cite{coville20}, we have the following conclusions.

\begin{proposition}\label{l2.3}
	Assume that the kernel $J$ satisfies ${\bf (J)}$ and $c\not=0$. Then 
		 $\lambda_p(\mathcal{L}^{J, c}_{(a,a+l)})$ is independent of $a\in\mathbb R$ and increasing and continuous in $l\in (0,\infty)$.
\end{proposition}

Denote
$$ \lambda^c_{\infty}:=\inf\limits_{\nu\in\mathbb{R}}\left[\int_{\mathbb{R}}J(x)e^{-\nu x}\,dx+c\nu\right].$$
Our next result shows that $ \lambda^c_{\infty}=\lim\limits_{l\to \infty}\lambda_p(\mathcal{L}^{J, c}_{(-l,l)})$, which appears to be the first explicit expression of this limit.
\begin{theorem}\label{th2.6}
	Assume that the kernel $J$ satisfies ${\bf (J)}$ and $c\not=0$. Then 
	\[
	\lim\limits_{l\to \infty}\lambda_p(\mathcal{L}^{J, c}_{(-l,l)})= \lambda^c_{\infty}.
	\]
\end{theorem}

\begin{proof}
	 We will use a variation of \eqref{minmax} and then complete the proof by further developing the ideas in the proof of Theorem \ref{l2.2}.  
	 
	 \medskip
	 
	 {\bf Step 1.} We prove the following variation of \eqref{minmax}:
	 	 \begin{align}\label{minmax1}
	\lambda_p(\mathcal{L}^{J, c}_{(-l,l)})=\sup\limits_{\phi\in \tilde X^c_{(-l,l)}}\inf\limits_{x\in (-l,l)}\frac{\mathcal{L}^{J,c}_{(-l,l)}[\phi](x)}{\phi (x)},
\end{align}
where
\[\begin{aligned}
\tilde X^c_{(-l,l)}:=\big\{ \phi\in C([-l,l]): &\  \phi\in X^c_{(-l,l)} \mbox{ except that $\ \phi'(x)$ is continuous  }\\
& \mbox{ in $(-l,l)$ with possible exception at  finitely  }\\
& \mbox{ many points where $\phi'(x)$ has jumping discontinuities}\big\},
\end{aligned}
\]
and for $\phi\in \tilde X^c_{(-l,l)}$ we always define
\[
\phi'(x)=\min\{\phi'(x-0),\phi'(x+0)\} \mbox{ for } x\in (-l,l).
\]
	
To simplify notations in the analysis below, we will write $\lambda_p=\lambda_p(\mathcal{L}^{J, c}_{(-l,l)})$.
Since $X^c_{(-l,l)}\subset \tilde X^c_{(-l,l)}$, we have
\[
\tilde\lambda:=\sup\limits_{\phi\in \tilde X^c_{(-l,l)}}\inf\limits_{x\in (-l,l)}\frac{\mathcal{L}^{J,c}_{(-l,l)}[\phi]}{\phi}\geq \sup\limits_{\phi\in  X^c_{(-l,l)}}\inf\limits_{x\in (-l,l)}\frac{\mathcal{L}^{J,c}_{(-l,l)}[\phi]}{\phi}=\lambda_p
\]
 
If $\tilde\lambda>\lambda_p$, then we can find $\lambda\in (\lambda_p, \tilde\lambda)$ and $\phi\in \tilde X^c_{(-l,l)}$ such that
\[
\frac{\mathcal{L}^{J,c}_{(-l,l)}[\phi](x)}{\phi(x)}\geq \lambda>\lambda_p \mbox{ in } (-l,l).
\]
We will show that this leads to a contradiction.

 For definiteness, from now on we assume $c>0$ and so $\phi(x)>0$ in $[-l, l)$ with $\phi(l)=0$. The case $c<0$ can be handled analogously.
 
Using the above inequality we obtain
\[
\int_{-l}^lJ(x-y)\phi(y)dy+c\phi'(x)\geq \lambda \phi(x) \mbox{ in } (-l,l).
\]
It follows that
\[
\liminf_{x\to l}\phi'(x)\geq -\frac 1 c \int_{-l}^lJ(l-y)\phi(y)dy=:-\sigma_0.
\]
Let $\psi>0$ be an eigenfunction associated to the principal eigenvalue $\lambda_p$:
\[
\int_{-l}^l J(x-y)\psi(y)dy+c\psi'(x)=\lambda_p \psi(x)\ \mbox{ for }  x\in (-l, l).
\]
Then
\[
\lim_{x\to l} \psi'(x)=-\frac 1 c \int_{-l}^lJ(l-y)\psi(y)dy<0.
\]
Therefore 
\[
v_\delta(x):=\psi(x)-\delta\phi(x)>0 \mbox{ in } [-l, l) \mbox{ for all small $\delta>0$.}
\]
Clearly $v_\delta(-l)<0$ for all large $\delta>0$. Therefore
\[
\delta_0:=\sup\{\delta: v_\delta(x)>0 \mbox{ in } [-l, l)\}
\]
is a well-defined positive number and 
\[
v_{\delta_0}(x)\geq 0 \mbox{ in } [-l, l].
\]
Moreover,
\begin{equation}\label{delta0}
\int_{-l}^lJ(x-y)v_{\delta_0}(y)dy+c v_{\delta_0}'(x)\leq\lambda_p\psi(x)-\delta_0 \lambda \phi(x)<\lambda v_{\delta_0}(x) \mbox{ in } (-l, l).
\end{equation}

We must have $v_{\delta_0}(x)\not\equiv 0$ for otherwise $\phi(x)\equiv \frac 1 {\delta_0}\psi(x)$, which leads to the contradiction 
\[
\lambda\leq \frac{\mathcal{L}^{J,c}_{(-l,l)}[\phi](x)}{\phi(x)}=\lambda_p<\lambda.
\]
 So we have two possibilities:

(a): $v_{\delta_0}(x)>0$ in $[-l, l)$ and $v_{\delta_0}(l)=0$,

(b): $\{x\in [-l, l): v_{\delta_0}(x)>0\}$ has a limiting  point $x_0\in [-l, l)$ with $v_{\delta_0}(x_0)=0$.\medskip

In case (b) we can find $y_n\in (-l, l)$ such that $y_n\to x_0$ as $n\to\infty$, $v_{\delta_0}'(y_n)$ exists and $v_{\delta_0}'(y_n)\geq 0$. Then take $x=y_n$ in \eqref{delta0} and let $n\to\infty$, and we obtain, due to $J(0)>0$,
\[
0<\int_{-l}^lJ(x_0-y)v_{\delta_0}(y)dy\leq 0,
\]
which is impossible.

In case (a) it is easy to see that $\limsup_{x\to l}v_{\delta_0}'(x)< 0$ leads to a contradiction to the definition of $\delta_0$. Therefore 
$\limsup_{x\to l}v_{\delta_0}'(x)\geq 0$. Choose $y_n\in (-l, l)$ such that $y_n\to l$ and $v_{\delta_0}'(y_n)\to \limsup_{x\to l}v_{\delta_0}'(x)$ as $n\to\infty$, and then take $x=y_n$ in \eqref{delta0} and let $n\to\infty$; we again arrive at the contradiction
\[
0<\int_{-l}^lJ(l-y)v_{\delta_0}(y)dy\leq 0.
\]
Therefore we must have $\tilde\lambda=\lambda_p$, as desired.\bigskip

{\bf Step 2.}  We show that
\begin{align}\label{lambda^c}
	\sup\limits_{\phi\in \tilde X^c_{(-l,l)}}\inf\limits_{x\in (-l,l)}\frac{\mathcal{L}^{J,c}_{(-l,l)}[\phi](x)}{\phi (x)}\to \lambda_\infty^c \mbox{ as } l\to\infty.
\end{align}

Denote $\lambda^c_l:=\lambda_p(\mathcal L^{J,c}_{(-l,l)})$. By \cite[Theorem 1.2]{coville20} we have
\begin{align*}\label{6b}
\lambda^c_l=\inf_{\phi\in X_{(-l,l)},\, \phi>0\,\text{in}\, (-l,l)}\ \sup_{x\in(-l,l)}\frac{\mathcal{L}^{J, c}_{(-l,l)}[\phi](x)}{\phi(x)}.
\end{align*}
 Hence, for any $\nu\in \mathbb{R}$, taking $\phi(x)=e^{\nu x}$ we obtain
\begin{align*}
\lambda^c_l&\leq \sup_{x\in(-l,l)}\frac{\int_{-l}^lJ(x-y)e^{\nu y}\,dy+c\nu e^{\nu x}}{e^{\nu x}}\\
&=\sup_{x\in(-l,l)}\int_{x-l}^{x+l}J(y)e^{-\nu y}\,dy+c\nu \leq \int_{\mathbb{R}}J(y)e^{-\nu y}\,dy +c\nu,
\end{align*}
which clearly implies $\lambda^c_l\leq \lambda^c_\infty$. It follows that $\lim_{l\to\infty}\lambda_l^c\leq \lambda^c_\infty$.
\medskip

	By \eqref{minmax1},
	\[
	\lambda^c_l\geq \inf\limits_{x\in (-l,l)}\frac{\mathcal{L}^{J,c}_{(-l,l)}[\phi](x)}{\phi (x)}
	\]
	for any $\phi\in \tilde X^c_{(-l,l)}$. We are going to construct a special function $\phi\in \tilde X^c_{(-l,l)}$ to prove $\lim_{l\to\infty} \lambda_l^c\geq \lambda^c_\infty$, by  modifying the function used in the proof of Theorem \ref{l2.2}.
	 
	 By Lemma \ref{uni}, there exists $\nu_0=\nu_0^c\in\mathbb R$ such that
	 \begin{equation}\label{stardrift}
	 \lambda_\infty^c=\int_{\mathbb{R}} J(x)e^{-\nu_0 x}\,dx +c\nu_0=\inf\limits_{\nu\in\mathbb{R}}\left[\int_{\mathbb{R}}J(x)e^{-\nu x}\,dx+c\nu\right].
	 \end{equation}
	 
	 The reasoning in Step 3 of the proof of Theorem \ref{l2.2} can be applied to our current situation to conclude  that it suffices to prove the desired inequality for $J$ with compact support. So from now on in this proof, we assume additionally that $J$ is compactly supported.
	 
Fix   some constant $M>0$ such that $\text{supp}(J)\subset (-M,M)$. Since $J(0)>0$, we easily see that
\[
\lim_{\nu\to\infty}\int_{-M}^0 J(y)e^{-\nu y}\,dy-|c|\nu =\lim_{\nu\to\infty}\int_0^M J(y)e^{\nu y}\,dy-|c|\nu=\infty,
\]
and
\[
\lim_{\nu\to\infty}\int_{-M}^0 |y|J(y)e^{-\nu y}\,dy =\lim_{\nu\to\infty}\int_0^M yJ(y)e^{\nu y}\,dy=\infty.
\]
Therefore we can find a large positive number $\sigma$  such that 
{\small \begin{equation}\label{large-k}
\begin{cases}\displaystyle
\mu_\sigma:=\min\Big\{\int_0^M \!\!\!\!\!J(y)e^{-(\nu_0-\sigma)y}\,dy, \int_{-M}^0\!\!\! \!\!J(y)e^{-(\nu_0+\sigma)y}\,dy\Big\}\!\!>\!\!|c|(\nu_0\!+\!\sigma)\!+\! \lambda^c_{\infty}\!+ \!\frac{c}{M},\medskip\\
\displaystyle \min\Big\{\int_0^M yJ(y)e^{-(\nu_0-\sigma)y}\,dy, \int_{-M}^0 |y|J(y)e^{-(\nu_0+\sigma)y}\,dy\Big\} > c.
\end{cases}
\end{equation}}
For any given $\epsilon>0$, there exists a large positive integer $k$ (depending on $\epsilon$) such that
\begin{equation}\label{less}
0<(1-e^{-\frac\sigma k M})\max\Big\{\int_{\mathbb R} J(y) e^{-\nu_0y}\,dy, \int_0^M \!\!\!\!\!J(y)e^{-(\nu_0-\sigma)y}\,dy, \int_{-M}^0\!\!\! \!\!J(y)e^{-(\nu_0+\sigma)y}\,dy \Big\}<\epsilon.
\end{equation}

For  $l>l_{\epsilon}:=4M(k+1)$, clearly we have $M<\frac 14 \frac{l}{k+1}$. Then  define, for $j=0, 1,2,..., k$,
\[
I_j^+:=\Big[\frac{j}{k+1}l,\ \frac{j+1}{k+1}l\Big),\ I_j^-:=\Big(-\frac{j+1}{k+1}l,\ -\frac{j}{k+1}l\Big],\ I_{k+1}^+:=[l, 2l),\ I_{k+1}^-:=(-2l, -l],
\]
and
\begin{equation*} \delta_0:=\frac\sigma k,\ \ 
\varphi(x):=\left\{
\begin{array}{ll}
\displaystyle    e^{(\nu_0-j\delta_0)x+\delta_0\Sigma_{i=1}^j \frac i{k+1}l}, & x\in I^+_j,\ j=1,..., k,\medskip\\
    e^{\nu_0 x}, & x\in I_0^+\cup I_0^-, \\
   \displaystyle  e^{(\nu_0+j\delta_0)x+\delta_0\Sigma_{i=1}^j \frac i{k+1}l}, & x\in I^-_j,\ j=1,..., k.
  \end{array}
\right.
\end{equation*}
Obviously, $\varphi$ is  continuous  and piecewise smooth on the interval $(-l,l)$.
 We then extend $\varphi$ to $(-2l, 2l)$ by defining
\[
\varphi(x)=\begin{cases}\frac {2l-x} le^{(\nu_0-k\delta_0)x+\frac k2 \delta_0 l } &\mbox{ for } l\leq x< 2l,\\
e^{(\nu_0+ k\delta_0)x+\frac k2 \delta_0 l} &\mbox{ for } -2l< x\leq -l,
\end{cases} \mbox{ \ \ \ when } c>0,
\]
\[
\varphi(x)=\begin{cases}\frac{2l+x}le^{(\nu_0+k\delta_0)x+\frac k2 \delta_0 l } &\mbox{ for } -2l\leq x\leq -l,\\
e^{(\nu_0-k\delta_0)x+\frac k2 \delta_0l} &\mbox{ for } l\leq x\leq 2l,
\end{cases} \mbox{ \ \ \ when } c<0.
\]
Hence $\varphi\in \tilde X^c_{(-2l,2l)}$, and to show the desired estimate   it suffices to show that 
\begin{align}\label{6}
E_l(x)>\lambda^c_{\infty}-\epsilon \mbox{ for } x\in (-2l, 2l),\ l> l_\epsilon.
\end{align}
where 
\[
E_l(x):=\frac{\mathcal{L}^{J,c}_{(-2l,2l)}[\varphi](x)}{\varphi(x)}=\frac{\displaystyle\int_{x-2l}^{x+2l}J(y)\varphi(x-y)\,dy+c\varphi'(x)}{\varphi(x)}.
\]

Next we estimate $E_l(x)$ according to the location of $x$ in $ (-2l, 2l)$. Firstly we notice that  
\[
[x-2l, x+2l]\supset [-M, M] \mbox{ for $x\in [-2l+M, 2l-M]$.}
\]
 
 For definiteness, we only consider the case $c>0$; the case $c<0$ can be treated analogously.\medskip
	
	We will check each of the six cases in the proof of Theorem \ref{l2.2} for the new situation here, plus four new cases.
	
\underline{\bf Case 1}.  $0\leq x< \frac{l}{k+1}-M$.

For $x$ in this range we have 
\[
\mbox{$-\frac{l}{k+1}< x-M\leq x+M< \frac l{k+1}$.}
\]
 Therefore
$$E_l(x)=\frac{\int_{-M}^{M}J(y)\varphi(x-y)\,dy+c\varphi'(x)}{\varphi(x)}=\int_{-M}^{M}J(y)e^{-\nu_0y}\,dy+c\nu_0=\lambda_{\infty}^c.$$

\underline{\bf Case 2}.  $\frac{l}{k+1}-M\leq  x< \frac{l}{k+1}$.

For $x$ in this range by  the definition of $\varphi$,  \eqref{stardrift} and \eqref{less}, we obtain 
\begin{align*}
E_l(x) &=\frac{\int_{-M}^{M}J(y)\varphi(x-y)\,dy+c\varphi'(x)}{\varphi(x)}\\
&=\int_{x-\frac{l}{k+1}}^{M}J(y)e^{-\nu_0y}\,dy+e^{\delta_0(\frac{l}{k+1}-x)}\int_{-M}^{x-\frac{l}{k+1}}J(y)e^{-(\nu_0-\delta_0)y}\,dy+c\nu_0\\
&\geq \int_{x-\frac{l}{k+1}}^{M}J(y)e^{-\nu_0y}\,dy+\int_{-M}^{x-\frac{l}{k+1}}J(y)e^{-(\nu_0-\delta_0)y}\,dy+c\nu_0\\
&=\int_{-M}^{M}J(y)e^{-\nu_0y}\,dy+c\nu_0+\int_{-M}^{x-\frac{l}{k+1}}J(y)e^{-\nu_0y}(e^{\delta_0y}-1)\,dy\\
&\geq \lambda^c_{\infty}-\int_{-M}^{0}J(y)e^{-\nu_0 y}(1-e^{\delta_0 y})\,dy\\
&\geq\lambda^c_{\infty}-(1-e^{-\delta_0 M})\int_{-M}^{0}J(y)e^{-\nu_0 y}\,dy>\lambda^c_{\infty}-\epsilon.
 \end{align*}

\underline{\bf Case 3}.  $\frac{jl}{k+1}\leq  x< \frac{jl}{k+1}+M$, $1\leq j\leq k$.

By the definition of $\varphi$ and \eqref{less}, we obtain
\begin{align*}
E_l(x) &=\frac{\int_{-M}^{M}J(y)\varphi(x-y)\,dy+c\varphi'(x)}{\varphi(x)}\\
&=\frac{\int_{x-\frac{jl}{k+1}}^{M}J(y)\varphi \Big|_{I_{j-1}^+}(x-y)\,dy+\int_{-M}^{x-\frac{jl}{k+1}}J(y)\varphi \Big|_{I_{j}^+}(x-y)\,dy+c\varphi'\Big|_{I_j^+}(x)}{\varphi\Big|_{I_j^+}(x)}\\
&\geq e^{\delta_0(x-\frac{jl}{k+1})}\int_{x-\frac{jl}{k+1}}^{M}J(y)e^{-(\nu_0-j\delta_0+\delta_0)y}\,dy+\int_{-M}^{x-\frac{jl}{k+1}}J(y)e^{-(\nu_0-j\delta_0)y}\,dy+c(\nu_0-j\delta_0)\\
&\geq \int_{x-\frac{jl}{k+1}}^{M}J(y)e^{-(\nu_0-j\delta_0+\delta_0)y}\,dy +\int_{-M}^{x-\frac{jl}{k+1}}J(y)e^{-(\nu_0-j\delta_0)y}\,dy+c(\nu_0-j\delta_0)\\
&=\int_{-M}^{M}J(y)e^{-(\nu_0-j\delta_0)y}\,dy+c(\nu_0-j\delta_0)-(1-e^{-\delta_0M})\int_{0}^{M}J(y)e^{-(\nu_0-k\delta_0)y}\,dy\\
&>\lambda^c_{\infty}-\epsilon \mbox{ by the minimality property of $\nu_0$.}
 \end{align*}

\underline{\bf Case 4}.  $\frac{jl}{k+1}+M\leq  x< \frac{(j+1)l}{k+1}-M$, $1\leq j\leq k$.

Making use of the definition of $\varphi$ and \eqref{stardrift}, we have
$$E_l(x)=\frac{\int_{-M}^{M}J(y)\varphi \Big|_{I_j^+}(x-y)\,dy+c\varphi' \Big|_{I_j^+}(x)}{\varphi \Big|_{I_j^+}(x)}
= \int_{-M}^{M}J(y)e^{-(\nu_0-j\delta_0)y}\,dy+c(\nu_0-j\delta)>\lambda^c_{\infty}.$$

\underline{\bf Case 5}.   $\frac{(j+1)l}{k+1}-M\leq  x<\frac{(j+1)l}{k+1}$, $1\leq j\leq k-1$.

Using  the definition of $\varphi$ and \eqref{stardrift}, \eqref{less}, we deduce
\begin{align*}
E_l(x)
&=\int_{x-\frac{(j+1)l}{k+1}}^{M}J(y)e^{-(\nu_0-j\delta_0)y}\,dy+e^{\delta_0(\frac{(j+1)l}{k+1}-x)}\int_{-M}^{x-\frac{(j+1)l}{k+1}}J(y)e^{-(\nu_0-j\delta_0-\delta_0)y}\,dy
+c(\nu_0-j\delta_0)\\
&\geq \int_{x-\frac{(j+1)l}{k+1}}^{M}J(y)e^{-(\nu_0-j\delta_0)y}\,dy+\int_{-M}^{x-\frac{(j+1)l}{k+1}}J(y)e^{-(\nu_0-j\delta_0-\delta_0)y}\,dy+c(\nu_0-j\delta_0)\\
&=\int_{-M}^{M}J(y)e^{-(\nu_0-j\delta_0)y}\,dy+c(\nu_0-j\delta_0)+\int_{-M}^{x-\frac{(j+1)l}{k+1}}J(y)e^{-(\nu_0-j\delta_0)y}(e^{\delta_0y}-1)\,dy\\
&>\lambda^c_{\infty}-(1-e^{-\delta_0M})\int_{-M}^{0}J(y)e^{-\nu_0y}\,dy>\lambda^c_{\infty}-\epsilon.
 \end{align*}

\underline{\bf Case 6}.   $x\in [l-M,l)$.

 Using   \eqref{large-k}, we obtain
\begin{align*}
E_l(x) &\geq \frac{\int_{x-l}^{M}J(y)\varphi(x-y)\,dy+c\varphi'(x)}{\varphi(x)}=\frac{\int_{x-l}^{M}J(y)\varphi\Big|_{I_k^+}(x-y)\,dy}{\varphi\Big|_{I_k^+}(x)}+c(\nu_0-k\delta_0)\\ 
&=\int_{x-l}^{M}J(y)e^{-(\nu_0-k\delta_0)y}\,dy+c(\nu_0-k\delta_0)\\
&\geq\int_{0}^{M}J(y)e^{-(\nu_0-k\delta_0)y}\,dy+c(\nu_0-k\delta_0)>\lambda^c_{\infty}.
 \end{align*}

To sum up, we have proved that $E_l(x)>\lambda^c_{\infty}-\epsilon$ for $x\in [0,l)$ and  $l>l_\epsilon$. Similarly, we can show $E_l(x)>\lambda^c_{\infty}-\epsilon$ for $x\in (-l,0]$ and  $l>l_\epsilon$. 

We next consider the case $x\in (-2l, -l]\cup [l, 2l)$.\medskip

\underline{\bf Case 7}.   $x\in [l,l+M)$.

We have, by using \eqref{large-k} in the last inequality below,
\begin{align*}
E_l(x) &=\frac{\int_{-M}^{M}J(y)\varphi(x-y)\,dy+c\varphi'(x)}{\varphi(x)}\\
&\geq \frac{\int^{x-l}_{-M}J(y)\varphi\Big|_{I_{k+1}^+}(x-y)\,dy}{\varphi\Big|_{I_{k+1}^+}(x)}+\frac{\int_{x-l}^{M}J(y)\varphi\Big|_{I_k^+}(x-y)\,dy}{\varphi\Big|_{I_{k+1}^+}(x)}+c(\nu_0- k \delta_0)-\frac {c}{2l-x}\\ 
&=\int^{x-l}_{-M}J(y)e^{-(\nu_0-k\delta_0)y}\,dy+\frac 1{2l-x}\int^{x-l}_{-M}J(y)ye^{-(\nu_0-k\delta_0)y}\,dy\\
&\ \ \ \ +\frac l{2l-x}\int_{x-l}^{M}J(y)e^{-(\nu_0-k\delta_0)y}\,dy +c(\nu_0-k\delta_0)-\frac {c}{2l-x}\\
&=\int^{M}_{-M}J(y)e^{-(\nu_0-k\delta_0)y}\,dy+c(\nu_0-k\delta_0)+\frac {c}{2l-x}\int^{x-l}_{-M}J(y)ye^{-(\nu_0-k\delta_0)y}\,dy\\
&\ \ \ \ +(\frac l{2l-x}-1)\int_{x-l}^{M}J(y)e^{-(\nu_0-k\delta_0)y}\,dy -\frac {c}{2l-x}\\
&\geq  \int^{M}_{-M}J(y)e^{-(\nu_0-k\delta_0)y}\,dy+c(\nu_0-k\delta_0)- \frac M{l-M} \int_{-M}^{0}J(y)e^{-(\nu_0-k\delta_0)y}\,dy -\frac {c}{l-M}\\
&\geq \int_{0}^{M}J(y)e^{-(\nu_0-k\delta_0)y}\,dy+c(\nu_0-k\delta_0) -\frac {c}M>\lambda_\infty^c,
 \end{align*}
where we have used the fact that $\frac M{l-M}\leq 1$.
 
 \underline{\bf Case 8}.   $x\in [l+M, 2l-M)$. 
 
 In this case we have, by using \eqref{large-k} in the last step,
 \begin{align*}
E_l(x) &=\frac{\int_{-M}^{M}J(y)\varphi(x-y)\,dy+c\varphi'(x)}{\varphi(x)}\\
&\geq \frac{\int^{M}_{-M}J(y)\varphi\Big|_{I_{k+1}^+}(x-y)\,dy}{\varphi\Big|_{I_{k+1}^+}(x)}+c(\nu_0- k \delta_0)-\frac {c}{2l-x}\\ 
&=\int^{M}_{-M}J(y)e^{-(\nu_0-k\delta_0)y}\,dy+\frac 1{2l-x}\int^{M}_{-M}J(y)ye^{-(\nu_0-k\delta_0)y}\,dy+c(\nu_0-k\delta_0)-\frac {c}{2l-x}\\
&\geq \int^{M}_{-M}J(y)e^{-(\nu_0-k\delta_0)y}\,dy+c(\nu_0-k\delta_0)-\frac {M}{2l-x}\int^{0}_{-M}J(y)e^{-(\nu_0-k\delta_0)y}\,dy-\frac {c}M\\
&\geq \int^{M}_{0}J(y)e^{-(\nu_0-k\delta_0)y}\,dy+c(\nu_0-k\delta_0) -\frac{c}{M}>\lambda_\infty^c,
 \end{align*}
where we have used the fact that $\frac {M}{2l-x}\leq 1$ for $x\in [l+M, 2l-M)$. 

 \underline{\bf Case 9}.   $x\in [2l-M, 2l)$.
 
 For $x$ in this range by  the definition of $\varphi$,  \eqref{stardrift} and \eqref{large-k}, we obtain 
\begin{align*}
E_l(x) &=\frac{\int_{x-2l}^{M}J(y)\varphi(x-y)\,dy+c\varphi'(x)}{\varphi(x)}\\
&=\frac{\int^{M}_{x-2l}J(y)\varphi\Big|_{I_{k+1}^+}(x-y)\,dy}{\varphi\Big|_{I_{k+1}^+}(x)}+c(\nu_0- k \delta_0)-\frac {c}{2l-x}\\ 
&=\int^{M}_{x-2l}J(y)e^{-(\nu_0-k\delta_0)y}\,dy+\frac 1{2l-x}\int^{M}_{x-2l}J(y)ye^{-(\nu_0-k\delta_0)y}\,dy+c(\nu_0-k\delta_0)-\frac {c}{2l-x}\\
&\geq \int_{0}^{M}J(y)e^{-(\nu_0-k\delta_0)y}\,dy+c(\nu_0- k \delta_0)+\frac 1{2l-x}\left[\int^{M}_{0}J(y)ye^{-\nu_0y}\,dy-c\right]\\
 &>\lambda_\infty^c,
 \end{align*}
where we have used the fact that $1+\frac {y}{2l-x}\geq 0$ for $y\in [x-2l, 0]$. 
 
  \underline{\bf Case 10}.   $x\in (-2l, -l]$.
  
   For $x$ in this range by  the definition of $\varphi$ and \eqref{large-k}, we obtain 
\begin{align*}
E_l(x) &=\frac{\displaystyle\int_{-M}^{\min\{M,\ x+2l\}}J(y)\varphi(x-y)\,dy+c\varphi'(x)}{\varphi(x)}\\
&=\frac{\displaystyle\int_{-M}^{\min\{M,\ x+2l\}}J(y)\varphi\Big|_{I_{k+1}^-}(x-y)\,dy}{\varphi\Big|_{I_{k+1}^-}(x)}+c(\nu_0+k \delta_0)\\ 
&=\int_{-M}^{\min\{M,\ x+2l\}}J(y)e^{-(\nu_0+k\delta_0)y}\,dy+c(\nu_0+k\delta_0)\\
&\geq \int^{0}_{-M}J(y)e^{-(\nu_0+k\delta_0)y}\,dy+c(\nu_0+ k \delta_0)>\lambda_\infty^c.
 \end{align*}

Thus we have proved the validity of \eqref{6},
and Step 2 is finished.
	 The desired conclusion follows directly from Step 1 and 2. The proof is now complete.
\end{proof}

	\section{Propagation dynamics of  \eqref{cau}}
	
	Throughout this section we assume that $J$ satisfies ${\bf (J)}$. We will make use of \eqref{l-to-infty} to prove Theorems \ref{c-l} and \ref{th1.3a}.

	For $c\in\mathbb R$, we consider the operator
	\[
	\mathcal L^c [\phi](x):=d\int_{\mathbb{R}}J(x-y)\phi(y)dy-d\phi(x)+f'(0)\phi+c\phi'(x).
	\]
	Clearly
	\[
	\lambda_p(\mathcal L^c)=d\lambda_p(\mathcal L^{J, c/d}_{\mathbb R})+f'(0)-d.
	\]

		\subsection{The sign of the principal eigenvalue $\lambda_p(\mathcal L^c)$}
	
\begin{proposition}\label{l9.2} Let $\mathcal L^c$ be defined as above, and $c_*^-, c_*^+$ be given by \eqref{c*+-} and \eqref{c*-gen}. Then
\[ \lambda_p(\mathcal{L}^c)\begin{cases}>0& \mbox{ if  $c\in (c_*^-,c_*^+)$},\\
		<0 & \mbox{ if } c\not\in [c_*^-, c_*^+],\\
		=0 & \mbox{ if } c\in \{c_*^-, c_*^+\}.
		\end{cases}
		\]
\end{proposition}
	\begin{proof} 	For fixed $c\in \mathbb R$  define
	\[
	A(\nu):=d\int_{\mathbb{R}}J(x)e^{-\nu x}\,dx-d+c\nu+f'(0).
	\]
	By Lemma \ref{uni} and Theorems \ref{l2.2} and \ref{th2.6}, there exists $\nu_0\in\mathbb R$ such that
		\begin{align*}
			A(\nu_0)=\inf\limits_{\nu\in\mathbb{R}}A(\nu)=d\lambda_p(\mathcal L^{J, c/d}_{\mathbb R})+f'(0)-d=\lambda_p(\mathcal L^c).
		\end{align*}
	Clearly,	
		\begin{align*}
	c-c_*^-=&	\inf\limits_{\nu<0}\frac{d\int_{\mathbb{R}}J(x)e^{\nu x}\,dx-d-c\nu+f'(0)}{-\nu}\\
	=&\inf\limits_{\nu>0}\frac{d\int_{\mathbb{R}}J(x)e^{-\nu x}\,dx-d+c\nu+f'(0)}{\nu}=\inf\limits_{\nu>0}\frac{A(\nu)}{\nu},\\
	c_*^+-c=&\inf\limits_{\nu>0}\frac{d\int_{\mathbb{R}}J(x)e^{\nu x}\,dx-d-c\nu+f'(0)}{\nu}\\
	=&\inf\limits_{\nu<0}\frac{d\int_{\mathbb{R}}J(x)e^{-\nu x}\,dx-d+c\nu+f'(0)}{-\nu}=\inf\limits_{\nu<0}\frac{A(\nu)}{-\nu}.
		\end{align*}
		Therefore $c<c_*^-$ implies $\inf\limits_{\nu>0}\frac{A(\nu)}{\nu}<0$ and hence $A(\nu_0)<0$ for some $\nu_0>0$, which leads to $\inf_{\nu\in\mathbb R}A(\nu)=\lambda_p(\mathcal L^c)<0$. Similarly, $c>c_*^+$ implies $\inf\limits_{\nu<0}\frac{A(\nu)}{-\nu}<0$ and hence $A(\nu_0)<0$ for some $\nu_0<0$, which again leads to $\inf_{\nu\in\mathbb R}A(\nu)=\lambda_p(\mathcal L^c)<0$.
		
		Now suppose $c\in (c_*^-, c_*^+)$. To prove the desired inequality it suffices to show $\inf_{\nu\in\mathbb R} A(\nu)=A(\nu_0)>0$.
Suppose, on the contrary, that $A(\nu_0)\leq 0$. Since $A(0)=f'(0)>0$, we must have  $\nu_0\neq 0$. 			
If $\nu_0<0$, then $0<c_*^+-c=\inf\limits_{\nu<0}\frac{A(\nu)}{-\nu}\leq \frac{A(\nu_0)}{-\nu_0}\leq 0$, and  if $\nu_0>0$, then $0<c-c_*^-=\inf\limits_{\nu>0}\frac{A(\nu)}{\nu}\leq \frac{A(\nu_0)}{\nu_0}\leq 0$. So $A(\nu_0)\leq 0$ always leads to a contradiction.   This proves $	A(\nu_0)>0$, and the desired inequality is proved.

From
\[
\lambda_p(\mathcal L^c)=d\lambda_p(\mathcal L^{J, c/d}_{\mathbb R})+f'(0)-d=d\lambda_\infty^{c/d}+f'(0)-d,
\]
we see that $\lambda_p(\mathcal L^c)$ depends continuously on $c$, and the conclusion $\lambda_p(\mathcal L^c)=0$ for $c\in\{c_*^-, c_*^+\}$ is a direct consequence of the inequalities just proved for $c\in (c_*^-, c_*^+)$ and $c\not\in [c_*^-, c_*^+]$.
The proof is finished.
	\end{proof}

\subsection{Proof of Theorem \ref{c-l}}
		
		Denote
		\[
	\mathcal L_l^c [\phi](x):=d\int_{-l}^lJ(x-y)\phi(y)dy-d\phi(x)+f'(0)\phi+c\phi'(x).
	\]
		
		\begin{lemma}\label{c=0}
		Suppose that  ${\bf (J)}$ and ${\bf (f_{KPP})}$ hold, and $V_0\in C([-l,l])$ is nonnegative and not identically $0$. If $c=0$, then \eqref{l-0} has a unique solution $V(t,x)$ and
		\[
		\lim_{t\to\infty} V(t,x)=\begin{cases} 0 &\mbox{ uniformly in $x\in[-l,l]$ if } \lambda_p(\mathcal L^0_l)\leq 0,\\
		V_l(x) &\mbox{ uniformly in $x\in[-l,l]$ if } \lambda_p(\mathcal L^0_l)> 0,
		\end{cases}
		\]
		where $V_l(x)$ is the unique positive stationary solution of \eqref{l-0}. Moreover, when $\lambda_p(\mathcal L^c)>0$ and hence
		$\lambda_p(\mathcal L^0_l)> 0$ for all large $l>0$, we have
		\[
		\lim_{l\to\infty} V_l(x)=1 \mbox{ uniformly for $x$ in any bounded interval of $\mathbb R$}.
		\]
		\end{lemma}
\begin{proof}
This follows from the same argument used in the proof of Propositions 3.5 and 3.6 of \cite{cao}, where the assumption that the kernel function $J$ is symmetric is not needed. We omit the details.
\end{proof}

It appears to us that when $c\not=0$, problem \eqref{l-0} has not been considered in the literature. We will prove that Lemma \ref{c=0} still holds for $c\not=0$ if $V_0\in\mathbb V_l^c$, by making use of an approach used in \cite{cao}, together with a comparison principle proved in \cite{duni20}.

If $V(t,x)$ solves \eqref{l-0}, then define
\[
U(t,x):=V(t, x-ct),\  g(t):=-l+ct, \ h(t):=l+ct,
\]
and it is easily seen that $U$ satisfies
\begin{equation}\label{U-l}
			\left\{
			\begin{array}{ll}
				\displaystyle U_t=d\int_{g(t)}^{h(t)}J(x-y)U(t,y)\,dy-dU+f(U), & t>0, x\in (g(t), h(t)), \\
				U(t,h(t))=0, & t>0 \mbox{ if } c>0,\\
				U(g(t), t)=0, & t>0 \mbox{ if } c<0,\\
				U(0,x)=V_0(x),  &x\in [-l, l].
			\end{array}
			\right.
		\end{equation}
		
\begin{lemma}\label{lem-U} 
For any $V_0\in\mathbb U_l^c$, \eqref{U-l} has a unique solution $U(t,x)$. Moreover, 	$U(t,x)$ is $C^1$ in $x$ for $x\in [g(t), h(t)]$ and $t\geq 0$ if $V_0\in \mathbb V^c_l$.	
\end{lemma}
\begin{proof} We will only consider the case $c>0$, as the proof for $c<0$ is analogous.

		Since $g'(t)=h'(t)=c>0$, by applying the maximum principle stated in Lemma 3.1 of \cite{duni20} (with $n=1$) we easily see that \eqref{U-l} can have at most one solution. To show the existence, we follow the approach in the proof of Lemma 2.3 of \cite{cao}. We cannot use this result directly since $g'(t)>0$, but the method there can be modified to treat \eqref{U-l}. 
		
		Fix any $T>0$, we are going to show that \eqref{U-l} has a unique solution defined for $t\in [0, T]$. As stated above, we will follow \cite{cao} and use a non-standard approach for this evolutionary problem to complete the proof of existence. We will do so in three steps. In Step 4, after the  existence is proved, we will show that the solution $U$ is $C^1$ in $x$ when $V_0\in \mathbb V^c_l$.
		
		 {\bf Step 1:} {\it A parametrized ODE problem.}

For given $x\in[-l,h(T)]$, define
\begin{equation}
\tilde u_0(x)=\left\{
\begin{aligned}
&0,& &x>l,\\
&V_0(x),& &x\in[-l,l]
\end{aligned}
\right.
~~~\text{and }~~~t_x=\left\{
\begin{aligned}
&0& &\mbox{ if } x\in[-l,l],\\
&(x-l)/c& &\mbox{ if $x\in(l,h(T)]$}.
\end{aligned}
\right.
\label{definition-of-t_x}
\end{equation}
Clearly $t_x=T$ for  $x=h(T)$, and $t_x<T$ for $x\in (l, h(T))$.
Define
\begin{align*}
&\Omega:=\left\{(t,x)\in\mathbb{R}^2: 0<t\leq T,~g(t)<x<h(t)\right\},\\
&\mathbb{X}=\Big\{\phi\in C(\overline\Omega)~:~\phi\ge0~\text{in}
~\Omega,~\phi(0,x)=V_0(x)~\text{for}~x\in [-l,l],~\\
& \hspace{6cm} \text{and} \;\; 
\phi(t,h(t))=0~\text{for }0\le t\le T\Big\}.
\end{align*}

 For any given $\phi\in
\mathbb{X}$, consider the following ODE initial value problem (with parameter $x$):
\begin{equation}\label{202}
\left\{
\begin{aligned}
 &v_t=d\int_{g(t)}^{h(t)}J(x-y)\phi(t,y)dy-dv(t,x)+\tilde f(v),& &t_x<t\le T,\\
&v(t_x,x)=\tilde u_0(x),& &x\in[-l,h(T)),
\end{aligned}
\right.
\end{equation}
where $\tilde f(v)=0$ for $v<0$, and $\tilde f(v)=f(v)$ for $v\geq 0$. 
Denote
$$
F(t,x,v)=d\int_{g(t)}^{h(t)}J(x-y)\phi(t,y)dy-dv
+\tilde f(v).
$$
Then, for any given $L>0$,
 and any $v_1,v_2
\in (-\infty, L]$, we have
$$
\Big|F(t,x,v_1)-F(t,x,v_2)\Big|\leq\Big|\tilde f(v_1)-\tilde f(v_2)\Big|
+d\Big|v_1-v_2\Big|,
$$
which implies that  the function $F(t,x,v)$ is Lipschitz continuous in $v$ for $v\in (-\infty, L]$
 with some Lipschitz
constant $K_L$, uniformly for $t\in [0, T]$ and $x\in [-l, h(T)]$. 
Additionally, $F(t,x,v)$ is continuous
in all its variables in this range.
Hence it follows from the Fundamental Theorem of
ODEs  that, for every fixed $x\in [-l, h(T))$, problem (\ref{202}) admits a
unique solution, denoted by $V_{\phi}(t,x)$, which is defined in some  interval $[t_x,T_x)$ of $t$.

We claim that $t\to V_\phi(t,x)$ can be uniquely extended to $[t_x, T]$. Clearly it suffices to show that if $V_\phi(t,x)$ is uniquely defined for $t\in [t_x, \tilde T]$ with $\tilde T\in (t_x, T)$, then
\begin{equation}\label{V-bd}
0\leq V_\phi(t,x)< L_\phi:=1+\|\phi\|_{C(\overline\Omega_T)}\mbox{ for } t\in (t_x, \tilde T].
\end{equation}

We first show that $V_\phi(t,x)<L_\phi$ for $ t\in (t_x, \tilde T]$. Arguing indirectly we assume that this inequality does not hold,
and hence, in view of $V_\phi(t_x,x)=\tilde u_0(x)\leq \|\phi\|_{C(\overline\Omega_T)} <L_\phi$, there exists some $t^*\in (t_x, \tilde T]$ such that
$V_\phi(t,x)<L_\phi$ for $t\in (t_x, t^*)$ and $V_\phi (t^*,x)=L_\phi$. It follows that $(V_\phi)_t(t^*,x)\geq 0$ and $\tilde f(V_\phi(t^*,x))\leq 0$
(due to $L_\phi>1$). We thus obtain from the differential equation satisfied by $V_\phi(t,x)$ that
\[
dL_\phi=dV_\phi(t^*,x)\leq d\int_{g(t^*)}^{h(t^*)}J(x-y)\phi(t^*,y)dy\leq d\|\phi\|_{C(\overline\Omega_T)}= d(L_\phi-1).
\]
It follows that $d\leq 0$. This contradiction proves our claim.

We now prove the first inequality in \eqref{V-bd}. Since
\[
\tilde f(v)=\tilde f(v)-\tilde f(0)\geq -K_L v \mbox{ for } v\in (-\infty, L],
\]
we have
\[
(V_\phi)_t\geq -K_L V_\phi +d\int_{g(t)}^{h(t)}J(x-y)\phi(t,y)dy\geq -K_L V_\phi \mbox{ for } t\in [t_x, \tilde t].
\]
Since $V_\phi(t_x,0)=\tilde u_0(x)\geq 0$, the above inequality immediately gives $V_\phi(t,x)\geq 0$ for $t\in [t_x, \tilde T]$.
 We have thus proved \eqref{V-bd}, and therefore the solution $V_\phi(t,x)$ of  \eqref{202} is uniquely defined for $t\in [t_x, T]$.

\medskip

 {\bf Step 2:} {\it A fixed point problem.}

Let us note that $V_{\phi}(0,x)=\tilde u_0(x)$ for $x\in [-l, l]$,
and $V_{\phi}(t,h(t))=0$ for $t\in [0, T]$. Moreover, by the continuous dependence of the unique solution on the initial value and on the parameters in the equation, we also see that $V_\phi(t,x)$ is continuous in $(t,x)\in\overline \Omega$. 

To use the Banach contraction theorem, we will need to first consider \eqref{U-l} for $t\in [0, s]$ with a sufficiently small $s>0$. For this purpose, we will denote $\Omega$ by $\Omega_s$ when $T=s$, and similarly $\mathbb X_s$ denotes $\mathbb X$ with $T=s$.
Then we accordingly define
 the mapping $\Gamma_s: \mathbb X_s\to \mathbb X_s$ by
\[
\Gamma_s \phi=V_{\phi},
\]
and so  $\Gamma_s \phi=\phi$ implies that $U=\phi(t,x)$ solves \eqref{U-l} for $t\in (0, s]$, and vice versa.

Checking the argument in Step 2 of the proof of Lemma 2.3 in \cite{cao}, we see that everything carries over to our current situation, and so for sufficiently small $s>0$,  say $s\in (0, \delta_1]$ with some small enough $\delta_1>0$, $\Gamma_s$ has a unique fixed point in $\mathbb X_s$.

\medskip

{\bf Step 3:} {\it Extension of the solution to $t=T$.}

This follows from the same argument as in Step 3 of the proof of Lemma 2.3 in \cite{cao}, where the process of Step 2 is repeated for $t\in [\delta_1, \delta_1+s]$ with $s\in (0,\delta_1]$, regarding $t=\delta_1$ as the new initial time, to obtain an extension of the solution to $t\in [0, 2\delta_1]$. This process can be repeated until $t=T$ is reached.

For use in Step 4 below, let us stress that $M$ and $\delta_1$ in Step 2 can be chosen the same when the process of Step 2 is repeated in Step 3 to extend the solution from $t\in [0, \delta_1]$ to $t\in [0, k\delta_1]$, $k=2,3,...$.
Since $T>0$ is arbitrary, we have thus proved that \eqref{U-l}  has a unique solution. 

To complete the proof of the lemma, it remains to finish the next step.

\medskip

{\bf Step 4:} We show that $U$ is $C^1$ in $x$ if $V_0\in \mathbb V_l^c$. 

Firstly we define a subset of $\mathbb X$ given in Step 1 above by 
\[
\mathbb X^1:=    \mathbb X\cap C^1(\overline \Omega).
\]
For $\phi\in \mathbb X^1$, 
\[
\int_{g(t)}^{h(t)}J(x-y)\phi(t,y)dy=\int_{x-h(t)}^{x-g(t)}J(z) \phi(t,x-z)dz \mbox{ is $C^1$ in $x$.}
\]

Now $V_0(x)$ is continuously differential in $[-l,l]$, so $V_\phi(t,x)$ obtained in Step 1 but with $\phi\in\mathbb X^1$, viewed as the solution of an ODE with parameter $x$, by standard ODE theory, 
is continuously differentiable in $x$ for $x\in [-l,l]$ and $t\in [0, \frac{x+l}c]$. Here the differentiation in $x$ at $x=\pm l$ is understood as one-sided differentiation. 
Similarly, since $t_x=\frac{x-l}c$ is smooth in $x$ for $x\geq l$, we see that $V_\phi(t,x)$ is continuously differentiable in $x$ for $x\geq l$ and $t\in [t_x, \frac{x+l}c]$. We show next that
\[
\lim_{x\nearrow l}\partial_xV_\phi(t,x)=\lim_{x\searrow l}\partial_xV_\phi(t,x).
\]
Denote $W(t,x):=\partial_xV_\phi(t,x)$. Then a direct computation gives, for $x\in [-l, h(T)]\setminus \{l\}$,
\begin{equation}\label{W}
W_t=J(x-g(t))\phi(g(t))+\int_{g(t)}^{h(t)}J(x-y)\phi_x(t,y)dy-dW+f'(V_\phi)W \mbox{ for } t_x<t<T.
\end{equation}

For $x=l+ct$ with $t>0$, from $V_\phi(t, l+ct)=0$ we deduce 
\[\begin{aligned}
\partial_xV_\phi(t, l+ct)&=-\frac 1 c\partial_t V_\phi(t,l+ct)\\
&=\lim_{x\nearrow l+ct}\frac{-1}c \left[d\int_{g(t)}^{h(t)}J(x-y) \phi(t,y)dy-dV_\phi(t,x)+f(V_\phi(t,x))\right]\\
&=-\frac d c \int_{g(t)}^{h(t)}J(h(t)-y)\phi(t, y)dy.
\end{aligned}
\]
It follows that, as $t$ decreases to 0,
\[
\partial_xV_\phi(t, l+ct)=-\frac d c \int_{g(t)}^{h(t)}J(h(t)-y)\phi(t, y)dy\to -\frac d c \int_{-l}^lJ(l-y)V_0(y)dy=V_0'(l).
\]
Since $\partial_x V_\phi(0,x)=V_0'(x)$ for $x\in [-l, l)$, the above identity implies that $W_0(x):=\partial_xV_\phi(t_x,x)$ is continuous over $[-l, h(T))$.
Moreover,
\begin{equation}\label{W_0}
|W_0(x)|\leq M_0:=\max\Big\{\|V_0'\|_{L^\infty([-l,l])}, \frac{2ld}c \|\phi\|_{L^\infty(\Omega)}\Big\}  \mbox{ for } x\in [-l, h(T)].
\end{equation}
Regarding $W(t,x)$ as the ODE solution of \eqref{W} with parameter $x$ and initial condition $W(t_x, x)=W_0(x)$, we can conclude from standard ODE theory that
$W\in C(\overline \Omega)$. Moreover we can use \eqref{W} to see that $\partial_tV_\phi\in C(\overline\Omega)$. Hence $V_\phi\in \mathbb X^1$.

We now define, for small $s>0$, $\Omega_s$,  $\mathbb X_s$ and $\Gamma_s:\mathbb X_s\to\mathbb X_s$ as in Step 2 above. Then following \cite{cao} as in Step 2 above, we can find $M>0$ independent of $s$ such that $\Gamma_s$ maps $\mathbb X_s^M:=\{\phi\in \mathbb X_s: \|\phi\|_\infty\leq M\}$ into itself and is a contraction mapping provided that $s\in (0, \delta_1]$ for some $\delta_1>0$ sufficiently small. This yields a unique fixed point $\phi$ of $\Gamma_s$ in $\mathbb X_s^M$, which gives a unique solution $U=V_\phi$ of \eqref{U-l} for $0\leq t\leq s$. If $\phi\in \mathbb X^1$, then our analysis about $W=\partial_x V_\phi$ above shows that $V_\phi\in \mathbb X^1$ and thus $U=V_\phi$ is continuously differential in $t$ and $x$ for $x\in [-l, h(s)]$, $t\in [0, s]$.

We show next that by shrinking $s>0$ further if necessary, $\Gamma_s$ has a fixed point  $\psi\in \mathbb X^M_s\cap \mathbb X^1$. If this is proved, then since $\Gamma_s$ has a unique fixed point $\phi\in \mathbb X_s^M$, we necessarily have $\psi=\phi$ and thus $V_\phi=\phi$ is in particular continuously differentiable in $x$  for $t\in [0,s]$, $x\in [g(t), h(t)]$.

For $\phi\in \mathbb X_s^M$, define
\[
C_\phi(t,x):=\int_{t_x}^t[d-f'(V_\phi(\tau,x))]d\tau \mbox{ for } (t,x)\in\overline \Omega_s.
\]
Since $0<s\ll 1$, we have
\[
|C_\phi(t,x)|\leq s (d+\max_{\xi\in [0,M]}|f'(\xi)|)\leq c_M:=d+\max_{\xi\in [0,M]}|f'(\xi)| \mbox{ for } (t,x)\in\overline \Omega_s.
\]
Using \eqref{202} and $V_\phi\in \mathbb X^M_s$ we easily deduce
\begin{equation}\label{W_t}
|\partial_tV_\phi(t,x)|\leq \tilde c_M:=2dlM+d+\max_{\xi\in [0,M]}|f(\xi)| \mbox{ for } (t,x)\in \Omega_s.
\end{equation}

Define
\[
 M_1:= 2e^{c_M}\big(M_0+\|J\|_\infty M\big), \ \mathbb X^1_{M, M_1}:=\{\phi\in \mathbb X^M_s\cap\mathbb X^1: |\phi_t|\leq \tilde c_M, \   |\phi_x|\leq M_1 \mbox{ in } \Omega_s\}.
\]
 We show that, for all small $s>0$ (independent of $M_1$),
\[
 \phi\in \mathbb X^1_{M, M_1}\implies V_\phi\in \mathbb X^1_{M, M_1}.
\]

We already know that $V_\phi\in \mathbb X_s^M$ for such $\phi$ from Step 2. By \eqref{W_t}, we also have 
\[
|\partial_tV_\phi(t,x)|\leq \tilde c_M  \mbox{ in } \Omega_s.
\]
It remains to show 
\[
 |\partial_xV_\phi(t,x)|\leq M_1 \mbox{ in } \Omega_s.
\]
By \eqref{W} we have
\[
[e^{C_\phi(t,x)}W]_t=e^{C_\phi(t,x)}[J(x-g(t))\phi(g(t))+\int_{g(t)}^{h(t)}J(x-y)\phi_x(t,y)dy].
\]
Therefore, making use of \eqref{W_0}, we obtain
\[
\begin{aligned}
|W(t,x)|&=e^{-C_\phi(t,x)}\big|W_0(t_x,x)+\int_{t_x}^t e^{C_\phi(\tau,x)}\Big[J(x-g(\tau))\phi(g(\tau))+\int_{g(\tau)}^{h(\tau)}J(x-y)\phi_x(\tau,y)dy\Big]d\tau\big|\\
&\leq e^{c_M}\left[|W_0(t_x,x)|+se^{c_M}\big(\|J\|_\infty M+2l\|J\|_\infty M_1\big)\right]\\
&\leq e^{c_M}\big(M_0+\|J\|_\infty M\big)+se^{c_M}2l\|J\|_\infty M_1\\
&\leq M_1/2+M_1/2= M_1 \mbox{ for all } (t,x)\in\Omega_s,
\end{aligned}
\]
provided that
\begin{equation}\label{s<}
0<s \leq \delta_2:=\frac 12(e^{c_M}2l\|J\|_\infty)^{-1}.
\end{equation}
This and \eqref{W_t} indicate that, for $0<s\leq\delta_0:=\min\{\delta_1,\delta_2\}$,  $\Gamma_s$ maps $\mathbb X^1_{M, M_1}$ into itself. 

Clearly $\mathbb X^1_{M, M_1}$ is a closed convex subset of the Banach space $\mathbb Y:= C^1(\overline \Omega_s)$ with norm
\[
\|\psi\|_{\mathbb Y}:=\|\psi\|_{C(\overline \Omega_s)}+\|\psi_t\|_{C(\overline \Omega_s)}+\|\psi_x\|_{C(\overline \Omega_s)}.
\]
 From \eqref{202} and \eqref{W}, we easily see that  $\Gamma_s$ is continuous over $\mathbb X^1_{M, M_1}$. If we can further prove that
$\Gamma_s(\mathbb X^1_{M, M_1})$ is compact in $\mathbb Y$, then we can apply the Schauder fixed point theorem to conclude that $\Gamma_s$ has a fixed point in $\mathbb X^1_{M, M_1}$, as desired.

We now examine $V_\phi(t,x)$ for given $\phi\in \mathbb X_{M, M_1}^1$. Clearly $v=V_\phi$ satisfies \eqref{202} with $T$ replaced by $s$. The right hand side of \eqref{202}  with $v=V_\phi$ and $T=s$ indicates that $\partial_t V_\phi(t,x)$ has a bound in $C(\overline\Omega_s)$ which is independent of $\phi\in \mathbb X^1_{M, M_1}$. We already know that $|\partial_x V_\phi(t,x)|\leq M_1$ for $(t,x)\in\Omega_s$. Therefore
$V_\phi(t,x)$ is equicontinuous  over $\Omega_s$ for $\phi\in \mathbb X^1_{M, M_1}$. Now we may use the right hand side of \eqref{202} again to conclude that $\partial_tV_\phi$ is
equicontinuous  over $\Omega_s$ for $\phi\in \mathbb X^1_{M, M_1}$.

We next show that $W(t,x)=\partial_x V_\phi(t,x)$ is also equicontinuous  over $\Omega_s$ for $\phi\in \mathbb X^1_{M, M_1}$.
We will use the expression
\[
W(t,x)=e^{-C_\phi(t,x)}\Big\{W_0(t_x,x)+\int_{t_x}^t e^{C_\phi(\tau,x)}\Big[J(x-g(\tau))\phi(g(\tau))+\int_{g(\tau)}^{h(\tau)}J(x-y)\phi_x(\tau,y)dy\Big]d\tau\Big\}.
\]
From the definition of $C_\phi(t,x)$ and the equicontinuity of $V_\phi$ we see that $C_\phi$ is equicontinuous  over $\Omega_s$ for $\phi\in \mathbb X^1_{M, M_1}$.
Since
\[
W_0(t_x,x)=\begin{cases} V_0'(x), & x\in [-l, l],\\
-\frac d c \int_{x-2l}^{x}J(x-y)\phi(t_x, y)dy, & x\in (l, h(s)],
\end{cases}
\]
we see that $W_0(t_x,x)$ is equicontinuous in $x$ over $[-l, h(s)]$ for $\phi\in \mathbb X^1_{M, M_1}$. Now clearly the mapping
\[
(t,x)\to\int_{t_x}^t e^{C_\phi(\tau,x)}\Big[J(x-g(\tau))\phi(g(\tau))+\int_{g(\tau)}^{h(\tau)}J(x-y)\phi_x(\tau,y)dy\Big]d\tau
\]
is equicontinuous  over $\Omega_s$ for $\phi\in \mathbb X^1_{M, M_1}$. We may now use the above expression of $W(t,x)$ to conclude that it is indeed equicontinuous  over $\Omega_s$ for $\phi\in \mathbb X^1_{M, M_1}$. Thus  $V_\phi(t,x)$, $\partial_tV_\phi(t,x)$ and $\partial_xV_\phi(t,x)$ are all uniformly bounded and equicontinuous over $\Omega_s$ for $\phi\in \mathbb X^1_{M, M_1}$. This implies that 
\[
\Gamma_s(\mathbb X^1_{M, M_1})=\{V_\phi: \phi\in \mathbb X^1_{M, M_1}\}
\]
is pre-compact in $\mathbb Y$. Since this set is closed (due to the continuity of $\Gamma_s$ and the closedness of $\mathbb X^1_{M, M_1}$), it must be a compact set of $\mathbb Y$, as desired. 

We now observe that $\delta_0=\min\{\delta_1,\delta_2\}$ is independent of $M_1$, and the unique solution of \eqref{U-l} is continuously differentiable in $(t,x)$ in $\overline\Omega_{\delta_0}$.

Let us also observe that $U(t,\cdot+ct)\in\mathbb V_l^c$ for all $t\in[0, \delta_0]$. We check this for $c>0$ only as the case $c<0$ is similar. From $U(t,l+ct)=0$ 
we deduce $U_t(t, l+ct)+cU_x(t, l+ct)=0$ and so
\[\begin{aligned}
U_x(t,l+ct)&=-\frac 1 cU_t(t,l+ct)\\
&=-\frac 1c\left[d\int_{g(t)}^{h(t)}J(l+ct-y)U(t,y)dy-dU(t,l+ct)+f(U(t,l+ct))\right]\\
&=-\frac dc\int_{-l}^lJ(l-y) U(t, ct+y)dy, \mbox{ as wanted.}
\end{aligned}
\]

Therefore,  in particular, $U(\delta_0, x+c\delta_0)$ belongs to $\mathbb V^c_l$. Regarding $t=\delta_0$ as the initial time of \eqref{U-l} we can show in the same way  (possibly with a different $M_1$) that the solution is continuously differentiable in $(t,x)$ for $t\in [\delta_0, 2\delta_0]$ and $x\in [g(t), h(t)]$.  By repeating this process we know that the solution of \eqref{U-l} is continuously differentiable in $(t,x)$ for all $t>0$ and $x\in [g(t), h(t)]$.
This completes Step 4 and hence the proof of the lemma. 
\end{proof}
\medskip

Let $U(t,x)$ be the solution of \eqref{U-l} with $V_0\in \mathbb V_l^c$ and $c\not=0$. By Lemma \ref{lem-U}, we see that $U_x$ exists and
hence $V(t,x):=U(t, x+ct)$ solves \eqref{l-0}. We are now ready to prove the following extension of Lemma \ref{c=0}.

\begin{lemma}\label{c-not-0} The conclusions in Lemma \ref{c=0} still hold when $c\not=0$  provided that $V_0\in \mathbb V_l^c$. 
\end{lemma}

\begin{proof} The existence and uniqueness of the solution $V(t,x)$ to \eqref{l-0} is guaranteed by Lemma \ref{lem-U} as indicated above. It remains to prove the asymptotic behaviour of $V(t,x)$ as $t\to\infty$. 

Let us first observe that $V(t,\cdot)\in\mathbb V_l^c$ for all $t>0$. We check this for $c>0$ only as the case $c<0$ is similar. From $V(t,l)=0$ for deduce $V_t(t,l)=0$ and so from the equation satisfies by $V$ we deduce
\begin{equation}\label{V(t,l)}
0=d\int_{-l}^l J(l-y)V(t,y)dy+cV_x(t, l).
\end{equation}

{\bf Step 1.}  We show that if $\lambda_p(\mathcal L^c_l)\leq 0$, then
\begin{equation}\label{U-to-0}
\lim_{t\to\infty} V(t,x)=0 \mbox{ uniformly for } x\in [-l, l].
\end{equation}
Again we only consider the case $c>0$. 
We first prove that any nonnegative stationary solution of \eqref{l-0} is identically $0$ when $\lambda_p(\mathcal L^c_l)\leq 0$.
Let $\phi(x)$ be a nonnegative stationary solution of \eqref{l-0}. Then $U(t,x)=\phi(x-ct)$ satisfies \eqref{U-l} with $U(0,x)=\phi(x)$. By Lemma 3.1 of \cite{duni20} applied to $U$ we see that either $\phi\equiv 0$ or $\phi(x)>0$ for $x\in [-l, l)$. 

We show that it is impossible to have $\phi(x)>0$ in $[-l, l)$ if $\lambda_p(\mathcal L^c_l)\leq 0$. Otherwise,   $f(\phi(x))<f'(0)\phi(x)$ for $x\in [-l, l)$
by ${\bf (f_{KPP})}$, and we have
\[
\mathcal L_l^c[\phi](x)=f'(0)\phi(x)-f(\phi(x))>0 \mbox{ for } x\in [-l, l).
\]
It follows that
\[
d\int_{-l}^lJ(x-y)\phi(y)dy+c\phi'(x)>[d-f'(0)] \phi(x) \mbox{ in } (-l,l),
\]
which implies
\[
\liminf_{x\to l}\phi'(x)\geq -\frac d c \int_{-l}^lJ(l-y)\phi(y)dy=:-\sigma_0.
\]
Let $\psi(x)$ be a positive eigenfunction associated with the principal eigenvalue $\lambda_p(\mathcal L_l^c)$:
\begin{equation}\label{eigenfn}
\mathcal L_l^c[\psi](x)=\lambda_p(\mathcal L^c_l)\psi(x),\ \psi(x)>0 \mbox{ for } x\in [-l, l),\ \psi(l)=0.
\end{equation}
Then
\[
\lim_{x\to l} \psi'(x)=-\frac d c \int_{-l}^lJ(l-y)\psi(y)dy<0.
\]
Therefore 
\[
v_\delta(x):=\psi(x)-\delta\phi(x)>0 \mbox{ in } [-l, l) \mbox{ for all small $\delta>0$.}
\]
Clearly $v_\delta(-l)<0$ for all large $\delta>0$. Therefore
\[
\delta_0:=\sup\{\delta: v_\delta(x)>0 \mbox{ in } [-l, l)\}
\]
is a well-defined positive number and 
\[
v_{\delta_0}(x)\geq 0 \mbox{ in } [-l, l].
\]
Moreover, since $\lambda_p(\mathcal L_l^c)\leq 0$, we have
\begin{equation}\label{delta00}
\begin{aligned}
d\int_{-l}^lJ(x-y)v_{\delta_0}(y)dy+c v_{\delta_0}'(x)&<[d-f'(0)]v_{\delta_0}(x)+\lambda_p(\mathcal L_l^c)\psi(x)\\
&\leq [d-f'(0)] v_{\delta_0}(x) \mbox{ in } (-l, l).
\end{aligned}
\end{equation}

We must have $v_{\delta_0}(x)\not\equiv 0$ for otherwise $\phi(x)\equiv \frac 1 {\delta_0}\psi(x)$, which leads to the contradiction 
\[
0<\mathcal{L}^{c}_{l}[\phi](x)=\lambda_p(\mathcal L_l^c)\phi(x)\leq 0 \mbox{ for } x\in (-l, l).
\]
 So we have two possibilities:

(a): $v_{\delta_0}(x)>0$ in $[-l, l)$ and $v_{\delta_0}(l)=0$,

(b): $\{x\in [-l, l): v_{\delta_0}(x)>0\}$ has a limiting  point $x_0\in [-l, l)$ with $v_{\delta_0}(x_0)=0$.\medskip

In case (b) we must have $v_{\delta_0}'(x_0):=\lim_{x\to x_0}v_{\delta_0}'(x)\geq 0$. Then letting $x\to x_0$ in \eqref{delta00}  we obtain, due to $J(0)>0$,
\[
0<\int_{-l}^lJ(x_0-y)v_{\delta_0}(y)dy\leq 0,
\]
which is impossible.

In case (a) the definition of $\delta_0$ implies that 
$\limsup_{x\to l}v_{\delta_0}'(x)\geq 0$. Choose $y_n\in (-l, l)$ such that $y_n\to l$ and $v_{\delta_0}'(y_n)\to \limsup_{x\to l}v_{\delta_0}'(x)$ as $n\to\infty$, and then take $x=y_n$ in \eqref{delta00} and let $n\to\infty$; we again arrive at the contradiction
\[
0<\int_{-l}^lJ(l-y)v_{\delta_0}(y)dy\leq 0.
\]
Therefore we must have $\phi(x)\equiv 0$, as desired.\bigskip

We are now ready to show \eqref{U-to-0}, by a standard approach. Let $\psi$ be a positive eigenfunction associated to $\lambda_p(\mathcal L_l^c)$ as given by \eqref{eigenfn}.
Then it is easily checked that $\psi\in \mathbb V_l^c$, and for any $M>0$, using ${\bf (f_{KPP})}$ we have
\[
\mathcal L_l^c[M\psi](x)=\lambda_p(\mathcal L_l^c)(M\psi(x))\leq 0\leq f'(0)(M\psi(x))-f(M\psi(x)).
\]
This implies that the unique solution of \eqref{l-0} with $V_0=M\psi\in \mathbb V_l^c$, denoted by $\bar V(t,x)$, is decreasing in $t$, and hence
\[
0\leq \bar V_\infty(x):=\lim_{t\to\infty} \bar V(t,x) \mbox{ exists};
\]
moreover, it is also easy to show that $\bar V_\infty$ is a nonnegative stationary solution of \eqref{l-0}. By our early conclusion on such stationary solutions we have $\bar V_\infty\equiv 0$.

On the other hand, for the unique solution $V(t,x)$ of \eqref{l-0}, it is easily seen that  for any $t>0$,
\[
V(t,\cdot)\in C^1([-l, l]) \mbox{ and $V(t,x)>0 $ in $[-l, l)$.}
\]
Since $\psi\in C^1([-l,l])$, $\psi(x)>0$ in $[-l, l)$ with $\psi(l)=0>\psi'(l)$, we can find $M>1$ large so that $V(1,x)\leq M\psi(x)$ in $[-l, l]$.
We may then use the comparison principle (Lemma 3.1 of \cite{duni20} applied to the equivalent problem in the form of \eqref{U-l}) to conclude that 
\[
0\leq V(1+t, x)\leq \bar V(t,x) \mbox{ for } t>0,\ x\in [-l, l].
\]
Letting $t\to\infty$ we obtain \eqref{U-to-0}.

\medskip

\noindent
{\bf Step 2.} We show that if $ \lambda_p(\mathcal L^c_l)> 0$, then
\[
		\lim_{t\to\infty} V(t,x)=V_l(x) \mbox{ uniformly in $x\in[-l,l]$,}
		\]
		where $V_l(x)$ is the unique positive stationary solution of \eqref{l-0}. 

Let $\psi$ be an eigenfunction of $\lambda_p(\mathcal L_l^c)$ given by \eqref{eigenfn}. Then for every small $\delta>0$, 
\[
\mathcal L_l^c[\delta\psi](x)=\lambda_p(\mathcal L_l^c)(\delta\psi(x)) \geq f'(0)(\delta\psi(x))-f(\delta\psi(x)).
\]
This implies that $\delta \psi$ is a lower solution of the stationary problem of \eqref{l-0}, and the unique solution of \eqref{l-0} with $V_0=\delta \psi\in\mathbb V_l^c$, denoted by $\underline V(t,x)$, is increasing in $t$. Moreover, let $w(t)$ be the solution of the ODE problem
\[
w'=f(w), \ w(0)=\delta\|\psi\|_\infty,
\]
then we can use the maximum principle Lemma 3.1 of \cite{duni20} to compare $\underline V(t, x-ct)$ with $w(t)$ over $\{(t,x): t\geq 0,\ x\in [-l+ct, l+ct]\}$ to conclude that $\underline V(t,x)\leq w(t)$. Clearly $w(t)\to 1$ as $t\to\infty$. It follows that
\[
  V_\infty(x):=\lim_{t\to\infty} \underline V(t,x) \mbox{ exists},
\]
and $0<V_\infty(x)\leq 1$ in $[-l, l)$.  It is also easy to show that $V_\infty$ is a  stationary solution of \eqref{l-0}. 

We show next that $V_\infty$ is the unique positive stationary solution of \eqref{l-0}. Let $V(x)$ be any positive stationary solution of \eqref{l-0}.
From the equation satisfied by $V$ and $V_\infty$ we easily deduce $V, V_\infty\in C^1([-l,l])$ with
\[
V(l)=V_\infty(l)=0,\ V'(l)<0,\ V_\infty'(l)<0,\ V(x)>0,\ V_\infty(x)>0 \mbox{ for } x\in [-l, l).
\]
Therefore
\[
\delta_0:=\sup\{\delta: V(x)-\delta V_\infty(x)>0 \mbox{ in } [-l, l)\}
\]
is a well-defined positive number and 
\[
V(x)\geq \delta_0 V_\infty(x) \mbox{ in } [-l, l].
\]
We claim that $\delta_0\geq 1$. Otherwise $\delta_0\in (0, 1)$ and by ${\bf (f_{KPP})}$ we have $f(\delta_0 V_\infty(x))\geq \delta_0 f(V_\infty(x))$ in $[-l, l]$.
Therefore $v(x):=V(x)-\delta_0 V_\infty(x)$ satisfies
\begin{equation}\label{delta0-v}
\begin{aligned}
d\int_{-l}^lJ(x-y)v(y)dy+c v'(x)&=dv(x)-f(V(x))+\delta_0 f(V_\infty(x))\\
&\leq dv(x)-f(V(x))+f(\delta_0V_\infty(x))\\
&=(d-c(x))v(x) \mbox{ in } (-l, l),
\end{aligned}
\end{equation}
where $c(x)=f'(\theta(x))$ for some $\theta(x)\in [\delta_0V_\infty(x), V(x)]$.

We must have $V(x)\not\equiv \delta_0V_\infty(x)$ for otherwise from the equations satisfied by $V$ and $V_\infty$ we deduce $f(\delta_0V_\infty(x))\equiv \delta_0f(V_\infty(x))$, which is a contradiction 
to ${\bf (f_{KPP})}$.
 So we have two possibilities:

(a): $v(x)>0$ in $[-l, l)$ and $v(l)=0$,

(b): $\{x\in [-l, l): v(x)>0\}$ has a limiting  point $x_0\in [-l, l)$ with $v(x_0)=0$.\medskip

We may then use \eqref{delta0-v} and the definition of $\delta_0$ to deduce a contradiction much as what we did by using \eqref{delta00} for $v_{\delta_0}$.
Therefore we must have $\delta_0\geq 1$ and hence $V(x)\geq V_\infty(x)$. Similarly we can prove $V_\infty(x)\geq V(x)$. Thus $V=V_\infty$ and the desired uniqueness result is proved.

By ${\bf (f_{KPP})}$ we see that for $M>1$, $f(MV_\infty(x))\leq Mf(V_\infty(x))$ in $[-l,l]$. This implies that $MV_\infty(x)$ is an upper solution of the stationary problem of \eqref{l-0}, and the unique solution  of \eqref{l-0} with $V_0=MV_\infty\in\mathbb V_l^c$, which we will denote by $\tilde V(t,x)$, is decreasing in $t$.
Therefore $\tilde V_\infty(x):=\lim_{t\to\infty} \tilde V(t,x)$ exists, and it can be easily shown that $\tilde V_\infty$ is a stationary solution of \eqref{l-0}. From \eqref{V(t,l)} we see that $V_x(1, l)<0$. Hence we can
 choose $\delta>0$ small and $M>1$ large such that
\[
\delta \psi(x)\leq V(1,x)\leq MV_\infty(x) \mbox{ in } [-l, l].
\]
Then by the comparison principle (deduced from Lemma 3.1 of \cite{duni20}) we obtain
\[
\underline V(t,x)\leq V(1+t,x)\leq \tilde V(t,x) \mbox{ for } t>0,\ x\in [-l, l].
\]
Letting $t\to\infty$ we obtain
\[
V_\infty(x)\leq \liminf_{t\to\infty} V(t,x)\leq\limsup_{t\to\infty}V(t,x)\leq \tilde V_\infty(x) \mbox{ uniformly for } x\in [-l, l].
\]
As $V_\infty$ is the unique positive stationary solution of \eqref{l-0}, we necessarily have $\tilde V_\infty\equiv V_\infty$, and the above inequalities thus imply
\[
		\lim_{t\to\infty} V(t,x)=V_\infty(x) \mbox{ uniformly in $x\in[-l,l]$.}
		\]
This finishes Step 2.
\medskip

\noindent
{\bf Step 3.}  We show that if $\lambda_p(\mathcal L^c)>0$ and hence
		$\lambda_p(\mathcal L^c_l)> 0$ for all large $l>0$, then the unique positive stationary solution $V_l(x)$ of \eqref{l-0} satisfies
		\[
		\lim_{l\to\infty} V_l(x)=1 \mbox{ uniformly for $x$ in any bounded interval of $\mathbb R$}.
		\]
This can be proved by the same technique as in the proof of Proposition 3.6 of \cite{cao}. We omit the details.

 The proof of the lemma is now complete.
\end{proof}

Theorem \ref{c-l} clearly follows directly from Lemmas \ref{c=0} and \ref{c-not-0}.

\subsection{Problem \eqref{U-l} revisited and  denseness of the set $\mathbb V_l^c$}

The following result on \eqref{U-l} follows directly from Theorem \ref{c-l} and the comparison principle.
\begin{corollary}\label{c-U}\!\!\!\footnote{ It is possible to show, by a simple perturbation argument,  that the conclusion here holds without the restriction \eqref{U_0-cond}, namely for any $U_0\in\mathbb U_l^c$ which is not identically 0.}
Let $U$ be the unique solution of \eqref{U-l} with $U(0,\cdot)=U_0\in\mathbb U_l^c$. If there exist $\underline V_0, \overline V_0\in \mathbb V_l^c$ such that
\begin{equation}\label{U_0-cond}
0\not\equiv \underline V_0(x)\leq U_0(x)\leq \overline V_0(x) \mbox{ for } x\in [-l, l],
\end{equation}
then
\[
\lim_{t\to\infty} U(t,x+ct)=\begin{cases} 0 &\mbox{ uniformly in $x\in[-l,l]$ if } \lambda_p(\mathcal L^0_l)\leq 0,\\
		V_l(x) &\mbox{ uniformly in $x\in[-l,l]$ if } \lambda_p(\mathcal L^0_l)> 0,
		\end{cases}
		\]
		where $V_l(x)$ is the unique positive stationary solution of \eqref{l-0}.
\end{corollary}

Our next result indicates that the condition on the initial function $U_0$ in Corollary \ref{c-U} is not difficult to satisfy.

\begin{lemma}\label{lem-V_0}
For any $U_0\in \mathbb U^c_l$ and small $\epsilon>0$, there exists $V_0\in \mathbb V^c_l$ such that
\[
|U_0(x)-V_0(x)|\leq \epsilon \mbox{ for } x\in [-l, l].
\]
\end{lemma}
\begin{proof} For definiteness, we assume that $c>0$ and hence $U_0(l)=0$.

Firstly we fix a function $\tilde V_0\in C^1([-l,l])$ such that
\[
\tilde V_0\geq 0, \ \tilde V_0(l)=U_0( l)=0,\  \tilde V_0'(l)=0 \mbox{ and } \|U_0-\tilde V_0\|_{C([-l,l])}<\epsilon/2.
\]

Secondly we let $\phi_\delta(x)$ be a family of $C^1([-l,l])$ functions satisfying
\[
0\leq \phi_\delta(x)\leq 1 \mbox{ in } [-l,l],\ \phi_\delta(l)=0,\ \phi_\delta'(l)=-1/\delta \mbox{ for } \delta\in (0, 1].
\]

We next show that there exists $\delta>0$ small and $m\in [0, \epsilon/2)$ such that $V_0:=\tilde V_0+m\phi_\delta$ has the desired properties.
We calculate 
\[\begin{aligned}
&V_0'(l)+\frac dc\int_{-l}^l J(l-y)V_0(y)dy\\
&=\tilde V_0'(l)+\frac dc\int_{-l}^l J(l-y)\tilde V_0(y)dy+m\Big[-\frac 1\delta+\frac dc\int_{-l}^l J(l-y)\phi_\delta(y)dy\Big].
\end{aligned}
\]
Clearly
\[
a:=\tilde V_0'(l)+\frac dc\int_{-l}^l J(l-y)\tilde V_0(y)dy=\frac dc\int_{-l}^l J(l-y)\tilde V_0(y)dy\geq 0.
\]
Since $\phi_\delta(x)\in [0, 1]$ we easily see that for sufficiently small $\delta>0$, 
\[
\displaystyle b_\delta:=-\frac 1\delta+\frac dc\int_{-l}^l J(l-y)\phi_\delta(y)dy<0.
\]
Therefore $V_0$ defined above belongs to $\mathbb V_l^c$ if $m=-a/b_\delta$.
Since $a/b_\delta=O(\delta)$ as $\delta\to 0$, we can fix $\delta>0$ small enough such that $0\leq -a/b_\delta<\epsilon/2$.
Then 
\[
|V_0(x)-U_0(x)|\leq |\tilde V_0(x)-U_0(x)|+m\phi_\delta(x)< \epsilon/2+m<\epsilon \mbox{ for } x\in [-l,l].
\]
The proof is complete.
\end{proof}

\subsection{Proof of Theorem \ref{th1.3a}}
		
		  Evidently, the solution of the ODE problem $\bar u'=f(\bar u)$ with $\bar u(0)=\|u_0\|_\infty$ converges to $1$ as $t\to\infty$. By the comparison principle, we have $u(t,x)\leq \bar u(t)$, and so
		\begin{align}\label{5.6a}
			\limsup_{t\to\infty}  u(t,x)\leq \lim_{t\to\infty} \bar u(t)=1 \mbox{ uniformly for } x\in\mathbb R.
		\end{align}
	
	Fix $c_*\in (c_*^-, c_*^+)$. By Proposition \ref{l9.2}, $\lambda_p(\mathcal L^{c_*})>0$ and hence $\lambda_p(\mathcal L_l^{c_*})>0$ for all large $l$.		
	\medskip
	
	\noindent
	{\bf Step 1.} 	We show that $u(t,x+c_*t)\to 1$ locally uniformly in $x\in \mathbb{R}$ as $t\to\infty$.
				
		By the strong maximum principle, we have $u(1,x)>0$ for $x\in\mathbb R$. For each large $l$ such that $\lambda_p(\mathcal L_l^{c_*})>0$, the continuous function $u(1,x)$ is positive on $[-l,l]$ and hence using Lemma \ref{lem-V_0} we can easily find some
		$V_0^l\in \mathbb V_l^c$ such that $0\not\equiv V_0^l(x)\leq u(1,x)$ in $[-l,l]$.
		Let $U^l(t,x)$ be the solution of \eqref{U-l} with $U_0(x)=V^l_0(x)$ and $c=c_*$. By Corollary \ref{c-U} we know that 
		\[
		\lim_{t\to\infty}U^l(t,x+c^*t)=V_l(x) \mbox{ uniformly for } x\in [-l,l],
		\]
		 $V_l$ being the unique positive stationary solution of \eqref{l-0} with $c=c_*$.
		 
		 Since $u(1, x)\geq V^l_0(x)$ for $x\in [-l,l]$, by the comparison principle (deduced from Lemma 3.1 of \cite{duni20}) we obtain $u(t+1,x)\geq U^l(t,x)$ for $t>0$ and $x\in [-l+c_*t,l+c_*t]$. It follows that
		 \[
		 \liminf_{t\to\infty} u(t,x+c_*t)\geq \lim_{t\to\infty}U^l(t-1,x+c_*t)=V_l(x+c_*) \mbox{ uniformly for } x+c_*\in [-l, l].
		 \]
		 By Theorem \ref{c-l}, we further have $V_l(x)\to 1$ as $l\to\infty$ locally uniformly in $\mathbb R$. This implies, by the above inequality,
		 \[
		 \liminf_{t\to\infty} u(t,x+c_*t)\geq 1 \mbox{ locally uniformly for } x\in \mathbb R.
		 \]
		 The desired conclusion clearly follows from this and \eqref{5.6a}.
		 \medskip
		 
		 \noindent
		 {\bf Step 2.} We show that $\lim_{t\to\infty} u(t,x)=1 $ uniformly for $x\in [a_1t, b_1t]$ provided that $[a_1, b_1]\subset (c_*^-, c_*^+)$.
		 
		 Let  $\epsilon>0$ be any given small constant. 	Then we can find $M>0$ such that
		\begin{align}\label{2.16m}
			\int_{-M}^{M}J(x) dx>1-\frac{1}{2d}\frac{f(1-\epsilon)}{1-\epsilon}.
		\end{align}
			Denote
		\begin{align*}
			\Omega_t:=[a_1t, b_1t],\quad \Omega_t^M:=[a_1t-M, b_1t+M].
		\end{align*}

		By choosing  $c_*=a_1$ and $c_*=b_1$ in Step 1, we see that for all large $t$, say $t\geq T>0$,
		it holds
		\begin{align}\label{u>}
			u(t,x)\geq 1-{\epsilon} \mbox{ for } \ x\in \Omega_t^M \setminus\Omega_t.
		\end{align}
			Moreover, it follows from the equation satisfied by $u$ that
		\begin{align*}
		u_t&\geq d\int_{\Omega^M_t}J(x-y)u(t,y)\,dy-du+f(u)\\
		&=d\int_{\Omega_t}J(x-y)u(t,y)\,dy-du+f(u)+d\int_{\Omega_t^M \setminus\Omega_t}J(x-y)u(t,y)\,dy \quad \mbox{ for }  t\geq T,\ x\in\Omega_t.
		\end{align*}

Choose $f_1\in C^1([0,\infty))$ satisfying  $f_1(0)=f_1(1-\epsilon)=0$, $f_1(s)>0$ for $s\in (0,1-\epsilon)$ and $f_1(s)<0$ for $s\in (1-\epsilon,\infty)$, and also
		\begin{align*}
   f_1(s)\leq f(s)-\frac{s}{2}\frac{f(1-\epsilon)}{1-\epsilon}=s\Big[\frac {f(s)}s-\frac{1}{2}\frac{f(1-\epsilon)}{1-\epsilon}\Big] \mbox{ for } \ s\in (0,1-\epsilon].
		\end{align*}
Then the solution $v(t)$ to
		\begin{equation*}
			\left\{
			\begin{array}{ll}
				v'=f_1(v), \\
				v(T)=\min\big\{\min_{x\in\Omega_T}u(T,x),1-\epsilon\big\}>0,
			\end{array}
			\right.
		\end{equation*}
		 is nondecreasing and  $\lim_{t\to \infty} v(t)=1-\epsilon$. Moreover,
		\begin{align}\label{2.17}
		f(v(t))-f_1(v(t))\geq \frac{v(t)}{2}\frac{f(1-\epsilon)}{1-\epsilon} \ \mbox{ for } \ t\geq T.
		\end{align}
		
	Define 	$v(t,x):=v(t)$. We claim that
		$$v_t\leq d\int_{\Omega_t}J(x-y)v(t,y)\,dy-dv+f(v)+d\int_{\Omega_t^M \setminus\Omega_t}J(x-y)u(t,y)\,dy\quad \mbox{ for } t\geq T,\, x\in\Omega_t.$$
		In fact,  using \eqref{u>}, \eqref{2.16m} and \eqref{2.17}, for $t\geq T$ and $x\in\Omega_t$, we have
		\begin{align*}
			& d\int_{\Omega_t}J(x-y)v(t,y)\,dy-dv+f(v)+d\int_{\Omega_t^M \setminus\Omega_t}J(x-y)u(t,y)\,dy\\
			&\geq  dv(t)\left(\int_{\Omega_t}J(x-y)\,dy-1+\int_{\Omega_t^M \setminus\Omega_t}J(x-y)\,dy\right)+f(v)\\
			&\geq  dv(t)\left(\int_{-M}^MJ(z)\,dz-1\right)+f(v)\geq- \frac{v(t)}{2}\frac{f(1-\epsilon)}{1-\epsilon}+f(v)\\
			&\geq  f_1(v)=v_t.
		\end{align*}
		 By the comparison principle we obtain
		$u(t,x)\geq v(t)$ for $x\in \Omega_t$ and $t\geq T$, and so
		\begin{align*}\label{5.7a}
			\liminf_{t\to\infty} \min_{x\in\Omega_t}u(t,x)\geq 1-\epsilon.
		\end{align*}
		In view of \eqref{5.6a},  the arbitrariness of $\epsilon$ in the above inequality implies the desired conclusion. 	

\medskip

\noindent
{\bf Step 3.} We show that 
\begin{align*}
	 \lim_{t\to \infty} u(t,x)=\begin{cases} 
	 0 \mbox{ uniformly for $x\leq a_2t$ provided that $c_*^->-\infty$ and $a_2<c_*^-$},\\
	 0 \mbox{ uniformly for $x\geq b_2t$ provided that $c_*^+<\infty$ and $b_2>c_*^+$}.
	 \end{cases}
	\end{align*}

This follows a standard approach. We provide the short proof for completeness.

Suppose $c_*^->-\infty$ and $a_2<c_*^-$.  Fix $\bar a_2\in (a_2, c_*^-)$. By the definition of $c_*^-$ we can find $\nu>0$ such that
\[
\frac{d\int_{\mathbb{R}} J(x)e^{-\nu x}\,dx - d + f'(0)}{-\nu}>\bar a_2.
\]
Since $u_0(x)$ is compactly supported, by choosing $M>0$ large enough we have
\[
M e^{\nu x}>u_0(x) \mbox{ for } x\in \mathbb R.
\]
Let $\bar u(t,x):=M e^{\nu(x-\bar a_2 t)}$. Direct calculation gives, in view of ${\bf (f_{KPP})}$,
\begin{align*}
\bar u_t-d\int_{\mathbb R} J(x-y)\bar u(t,y)dy+d\bar u-f(\bar u)&=\bar u\big[-\bar a_2\nu-d\int_{\mathbb R}J(y) e^{-\nu y}dy+d\big]-f(\bar u)\\
&\geq \bar u\big[-\bar a_2\nu-d\int_{\mathbb R}J(y) e^{-\nu y}dy+d-f'(0)\big]\\
&\geq 0.
\end{align*}
Therefore we can use the comparison principle to conclude that $0\leq u(t,x)\leq \bar u(t,x)$ for $t>0$ and $x\in\mathbb R$. Clearly $\bar u(t,x)\to 0$ as $t\to\infty$ uniformly for $x\leq a_2 t$. It follows that the same holds for $u(t,x)$. 

If $c_*^+<\infty$ and $b_2>a_*^+$, the proof is similar, and we omit the details. The proof of Theorem \ref{th1.3a} is now complete. \hfill $\Box$
\bigskip

\noindent {\bf Acknowledgements.}
Du and Ni were supported by the Australian Research Council. Part of this research was carried out during
Fang's visit of  the University of New England supported by the China Scholarship Council.

\end{document}